\documentclass{amsart}
\usepackage{geometry}   
\usepackage[dvipdfmx]{graphicx}

\usepackage{amsmath,amssymb,eucal,mathtools}
\usepackage[all]{xy}
\usepackage{ulem}
\usepackage{ytableau}

\usepackage{amssymb, amscd, epsfig, mathrsfs, amsmath,amsthm, color}

\usepackage{tikz}

\newcommand{\ep}{\varepsilon}

\newcommand{\sh}{\theta_k}

\newcommand{\pol}{\mathrm{pol}}

\newcommand{\SL}{\mathrm{SL}}

\newcommand{\C}{\mathbb{C}}
\newcommand{\Z}{\mathbb{Z}}

\newcommand{\res}{\mathrm{res}}

\newcommand{\loc}{\mathrm{loc}}
\newcommand{\Fl}{G/B}

\renewcommand{\O}{\mathcal{O}}

\newcommand{\G}{\mathfrak{G}}

\newcommand{\F}{\sigma}
\newcommand{\inv}{\omega}
\newcommand{\T}{\mathcal{T}}
\renewcommand{\r}{\mathfrak{r}}
\newcommand{\Kato}{\kappa}
\newcommand{\K}[3]{K(#1;#2;#3)}

\newcommand{\up}{\mathrm{up}}

\renewcommand{\top}{\mathrm{top}}
\newcommand{\mtop}{\mathrm{mtop}}
\newcommand{\Gr}{\mathrm{Gr}}

\newcommand{\Par}{\mathcal{P}}
\newcommand{\gt}[1]{\tilde{g}^{(#1)}}

\newcommand{\Waf}{\tilde{S}_{k+1}}
\newcommand{\Wafo}{\tilde{S}_{k+1}^0}
\newcommand{\Wafh}{\hat{S}_{k+1}}
\newcommand{\tfg}[1]{\tilde{\mathfrak{g}}^{(#1)}}
\newcommand{\fg}[1]{{\mathfrak{g}}^{(#1)}}

\newcommand{\Des}{\mathrm{Des}}
\newcommand{\gtf}[1]{\tilde{\mathfrak{g}}^{(#1)}}
\newcommand{\gcirc}[1]{\overset{\circ}{\mathfrak{g}}^{(#1)}}

\newcommand{\leqk}{\underset{k}{\leq}}
\newcommand{\A}{\mathcal{S}}

\usepackage{scalerel,stackengine}
\stackMath
\newcommand\reallywidecheck[1]{%
\savestack{\tmpbox}{\stretchto{%
  \scaleto{%
    \scalerel*[\widthof{\ensuremath{#1}}]{\kern-.6pt\bigwedge\kern-.6pt}%
    {\rule[-\textheight/2]{1ex}{\textheight}}%WIDTH-LIMITED BIG WEDGE
  }{\textheight}% 
}{0.5ex}}%
\stackon[1pt]{#1}{\scalebox{-1}{\tmpbox}}%
}
\numberwithin{equation}{section}
\theoremstyle{definition}
\newtheorem{thm}{Theorem}[section]

\newtheorem{lem}[thm]{Lemma}
\newtheorem{prop}[thm]{Proposition}
\newtheorem{cor}[thm]{Corollary}
\newtheorem{defn}[thm]{Definition}

\newtheorem{example}[thm]{Example}
\newtheorem{rem}[thm]{Remark}

\author{Takeshi Ikeda}
\address{Faculty of Science and Engineering, Waseda University, 3-4-1 Okubo, Shinjuku-ku, Tokyo 169-8555 Japan}
\email{gakuikeda@waseda.jp}
\author{Shinsuke Iwao}
\address{Faculty of Business and Commerce, Keio University, 4-1-1 Hiyoshi, Kohoku-ku, Yokohama-shi, Kanagawa 223-8521 Japan}
\email{iwao-s@keio.jp}
\author{Satoshi Naito}
\address{Department of Mathematics, Tokyo Institute of Technology, 2-12-1 Oh-Okayama, Meguro-ku, Tokyo 152-8551 Japan}
\email{naito@math.titech.ac.jp}

\title[Closed $k$-Schur Katalan functions]{Closed $k$-Schur Katalan functions as $K$-homology Schubert representatives
of the affine Grassmannian
}

\dedicatory{Dedicated to the memory of Bumsig Kim}

\date{\today}


\begin{document}
%\abstract{
\begin{abstract}
Recently, Blasiak-Morse-Seelinger
introduced symmetric functions called Katalan functions,
and proved that the $K$-theoretic $k$-Schur functions due to Lam-Schilling-Shimozono form  
a subfamily of the Katalan functions.  
They conjectured that another subfamily of Katalan functions called 
closed $k$-Schur Katalan functions are 
identified with the Schubert structure sheaves in the $K$-homology of the affine Grassmannian.
Our main result is a proof of this conjecture.

We also study a $K$-theoretic Peterson isomorphism that 
Ikeda, Iwao, and Maeno constructed, in a non-geometric manner, based on the unipotent solution of 
the relativistic Toda lattice of Ruijsenaars. 
We prove that the map sends a Schubert class of the quantum $K$-theory ring of the flag variety
to a closed $K$-$k$-Schur Katalan function up to an explicit factor related to 
a translation element with respect to an anti-dominant coroot. In fact, 
we prove the above map coincides with a map whose existence was conjectured by Lam, Li, Mihalcea, Shimozono, and proved by 
Kato, and more recently by Chow and Leung.
\end{abstract}
%}

\maketitle

\section{Introduction}
The study of $K$-theoretic Schubert calculus attracts much attention in the last few decades. 
In this paper, we focus on the $K$-theory version of the ``quantum equals affine'' phenomenon, which originally comes from an unpublished result by Peterson in 1997 for the case of (co)homology, and has been developed by many authors. 
See the textbook \cite{book} by 
 Lam, Lapointe, Morse, Schilling, Shimozono, and Zabrocki 
 on these topics.
 
Lam, Schilling, and Shimozono \cite{LSS} identifies 
the $K$-homology $K_*(\Gr)$ 
of the affine Grassmannian $\Gr=G(\C((t)))/G(\C[[t]])$ of $G=SL_{k+1},$
with a subring $\Lambda_{(k)}=\C[h_1,\ldots,h_k]$ of the ring of symmetric functions, where $h_i$ is the $i$-th complete symmetric function. 
In particular, the $K$-{\it theoretic $k$-Schur functions\/} $g_\lambda^{(k)}$, 
the $K$-$k$-{\it Schur functions\/} for short,
 were introduced in \cite{LSS}, which form 
a family of inhomogeneous symmetric functions in $\Lambda_{(k)}$. It was proved that the $K$-theoretic $k$-Schur functions, indexed by the partitions $\lambda$ with $\lambda_1\leq k$, are identified with a distinguished basis of 
$K_*(\Gr).$
These functions form a basis of $\Lambda_{(k)}$, and are indexed by the partitions $\lambda$ with $\lambda_1\leq k.$

Although the $K$-$k$-Schur functions can be characterized by 
a Pieri type formula (see Definition \ref{def:KkSchur}), there was no explicit combinatorial formula until recently. 
Blasiak, Morse, and Seelinger \cite{BMS}
proved a raising operator formula for the $K$-$k$-Schur functions.
In fact, they introduced a family of
inhomogeneous symmetric functions
called 
$K$-{\it theoretic Catalan functions}, {\it Katalan functions\/} for short, 
and proved that the $K$-theoretic $k$-Schur functions form a subfamily of the 
Katalan functions. 

The Katalan functions $K(\Psi;M;\gamma)$
are indexed by triples $(\Psi,M,\gamma)$, where 
$\Psi$ is an upper order ideal in the set $\Delta_\ell^+
:=\{\ep_i-\ep_j\;|\; 1\leq i<j\leq \ell\}$ of positive roots
of type $A_{\ell-1}$, $M$ is a multiset supported on $\{1,\ldots,\ell\}$, and $\gamma
\in \Z^\ell.$
For any root ideal $\mathcal{L}\subset \Delta_\ell^+,$ 
set $L(\mathcal{L}):=\bigsqcup_{(i,j)\in \mathcal{L}}\{j\},$
where $(i,j)$ is a short hand notation for $\ep_i-\ep_j.$ 
Let $\Par^k$ denote the set of all partitions $\lambda$ such that $\lambda_1\leq k$.
For $\lambda\in \Par^k$, let  $g^{(k)}_\lambda$  denote 
the corresponding $K$-$k$-Schur function (see \S\ref{sec:Kk} for the definition).
In \cite{BMS}, it was proved that
\begin{equation}
g_\lambda^{(k)}=
K(\Delta^k(\lambda);L(\Delta^{k+1}(\lambda));\lambda),\label{eq:BMS}
\end{equation}
where 
$$
\Delta^k(\lambda):=\{(i,j)\in \Delta_\ell^+\;|\;\lambda_i+j-i>k\},
$$ 
and $\ell\geq \ell(\lambda)$, the length of $\lambda$. 
As one of the consequences of \eqref{eq:BMS}, a long-standing conjecture by Morse \cite{M12} was verified
in \cite{BMS}:
for $\lambda\in \Par^k$,
\begin{equation}
\lambda_1+\ell(\lambda)\leq k+1
\Longrightarrow
g^{(k)}_\lambda=g_\lambda,\label{eq:M12}
\end{equation}
where $g_\lambda$ is the {\it dual stable Grothendieck polynomial\/}. 

In \cite{BMS}, they introduced another subfamily of Katalan functions, called {\it closed $k$-Schur Katalan functions}, defined for $\lambda\in \Par^k$ by 
$$\tilde{\mathfrak{g}}^{(k)}_\lambda:=K(\Delta^k(\lambda);L(\Delta^k(\lambda));\lambda).$$
It is conjectured that the closed $k$-Schur Katalan function is 
related to the function
\begin{equation}
\tilde{g}^{(k)}_\lambda
:=\sum_{\mu\leqk\lambda}{g}^{(k)}_\mu,\label{eq:defTak}
\end{equation}
where $\lambda,\mu\in \Par^k,$ and $\leqk $ denotes the order on $\Par^k$ induced by 
the Bruhat order on the affine symmetric group $\tilde{S}_{k+1}$ (see \S\ref{sec:aff}  for details) .
We call $\gt{k}_\lambda$ a {\it closed} $K$-$k$-{\it Schur function\/}.
These functions are essential in the $K$-homology Schubert calculus 
 because it is identified with the class of the Schubert structure sheaf $\O_\lambda^\Gr$ for the affine Grassmannian, {whereas $g^{(k)}_\lambda$ is identified with the class of ideal sheaf of the boundary of the Schubert variety; see
 \cite[Theorem 1]{LLMS} and \cite[Theorems 5.4 and 7.17(1)]{LSS}.} It should be noted that 
Takigiku \cite{Tak} proved a Pieri type formula for 
$\tilde{g}^{(k)}_\lambda$.
Another important 
result in \cite{Tak} is called the $k$-{\it rectangle factorization formula\/}. 
For $1\leq i\leq k$, define
\begin{equation*}
R_i= (\overbrace{i,\ldots,i}^{k+1-i}).
\end{equation*}
Takigiku showed
\begin{equation}
\tilde{g}_{R_i}\cdot \gt{k}_\lambda=\gt{k}_{\lambda\cup R_i},\label{eq:kfact}
\end{equation}
where $\lambda\cup R_i$ is the partition 
made by  combining the parts of $\lambda$ and those of $R_i$ and then sorting them.
This formula is natural from a geometric point of view (see \cite{LLMS}),
and plays an important role in our construction.

Let $\F$ be a ring automorphism 
of $\Lambda$ given by $\F(h_i)=\sum_{j=0}^{i}h_j\;(i\geq 1)$
with $h_0=1$.
We can now state the main result of this paper 
confirming a conjecture \cite[Conjecture 2.12 (a)]{BMS}
by Blasiak, Morse, and Seelinger, which enable us to have a better explicit knowledge of 
the structure sheaf $\O_\lambda^\Gr$.  
\begin{thm}\label{thm:main}
For $\lambda\in \Par^k$, 
we have
\begin{equation*}
\tilde{g}^{(k)}_\lambda=\F(\tilde{\mathfrak{g}}^{(k)}_\lambda).
\end{equation*}
\end{thm}

Knowing Takigiku's result \eqref{eq:kfact},
an immediate consequence is the following $k$-rectangle
factorization formula for $\tfg{k}_\lambda$, which 
was first proved by Seelinger \cite{See} by a more direct 
method involving generalized Katalan functions defined for 
arbitrary subsets of $\Delta_\ell^+$.

\begin{cor}(cf. \cite[Conjecture 2.12 (f)]{BMS})
For $1\leq i\leq k,$ we have
$$\tfg{k}_{R_i}\cdot \gtf{k}_\lambda=\gtf{k}_{\lambda\cup R_i}.$$
\end{cor}

\bigskip

Another motivation of our study is to clarify the connection between
a map called the $K$-{\it theoretic Peterson isomorphism\/} 
constructed by Ikeda, Iwao, and Meano \cite{IIM}, and 
a map whose existence was conjectured by 
Lam, Li, Mihalcea, and Shimozono \cite{LLMS}; this conjecture was proved by 
Kato \cite{Kato}, and more recently by 
Chow and Leung \cite[a]{CL} using different methods. 

Let $QK(\Fl)$ be the (small) quantum $K$-theory ring
of the flag variety $\Fl=\SL_{k+1}(\C)/B$, a deformation of 
the Grothendieck 
ring of coherent sheaves on $G/B$ studied by 
Givental and Lee \cite{GL} (see also the finiteness result \cite{ACTI} by 
Anderson, Chen, Tseng, and Iritani). 
This is 
a commutative associative algebra over 
the formal power series ring $\C[[Q]]:=\C[[Q_1,\ldots,Q_k]]$ in the variables $Q_i$,  
called the Novikov variables. For $w\in S_{k+1}$, the symmetric group of degree $k+1$,
the Schubert variety $\Omega^w$ in $G/B$
is defined to be $\overline{B_{-}wB/B}$, where $B_{-}$ is the opposite 
Borel subgroup. Let $\O^w_{\Fl}$ denote 
(the class of) the structure sheaf of $\Omega^w.$
As a $\C[[Q]]$-module,  $QK(\Fl)$
has a basis consisting of $\O^w_{\Fl}\;(w\in S_{k+1})$. 
Let $K_*(\Gr)$ be the $K$-homology group  
of the affine Grassmannian $\Gr,$ which has a ring structure 
by the Pontryagin product. 
The ring has a basis consisting of (the class of ) the Schubert structure sheaves
$\O_\lambda^\Gr$ indexed by the $k$-bounded partitions (see \cite{LSS} for more detailed description).
 
{We represent $QK(\Fl)$ as 
a quotient ring $A_{k+1}$ of the the polynomial
ring $\C[[Q]][z_1,\ldots,z_{k+1}]$ (see \S\ref{sec:QK} for details). 
According to Lam, Schilling, and Shimozono \cite{LSS}, we can identify
$K_*(\Gr)$ with $\Lambda_{(k)}$. More precisely, 
$\O_\lambda^\Gr$ is identified with $\F^{-1}(\tilde{g}_\lambda^{(k)})$
in our convention (see \S\ref{sec:Khom} for details). 
For any partition $\lambda$, let $g_\lambda\in \Lambda$
be the {dual stable Grothendieck polynomial\/} (see \S\ref{sec:dualstable}).  
We also set $\tilde{g}_\lambda:=\sum_{\mu\subset \lambda}g_\mu$, following Takigiku \cite{Tak1}.   
% We define a localization $K_*(\Gr)_\loc$ of $K_*(\Gr)
% \cong \Lambda_{(k)}$ as 
% $\Lambda_{(k)}[g_{R_i}^{-1},\tilde{g}_{R_i}^{-1}\;(1\leq i\leq k)].$
We let $A_{k+1}^\pol$ be a $\C[Q]$-algebra which is a polynomial version of 
$A_{k+1}$ (see \S\ref{sec:QK}).
% by the multiplicative set generated by $Q_1,\ldots,Q_k$.
There is a ring isomorphism~\cite{IIM}
given by : \begin{equation*}
\Phi_{k+1}: A_{k+1}^\pol[Q_i^{-1}(1\le i\le k)]\rightarrow 
\Lambda_{(k)}[\tau_i^{-1},(\tau_i^{+})^{-1}(1\le i\le k)],
%\cong K_*(\Gr)_\loc,
\end{equation*}
\begin{equation*}
\Phi_{k+1}
(z_i)=\frac{\tau_i\tau_{i-1}^+}{\tau_i^+\tau_{i-1}},\quad
\Phi_{k+1}(Q_i)=\frac{\tau_{i-1}\tau_{i+1}}{\tau_i^2},
\end{equation*}
where $$\tau_i=g_{R_i},\quad \tau_i^+=\tilde{g}_{R_i}$$ for $1\leq i\leq k$
and $\tau_0=\tau_{k+1}=\tau_0^+=\tau^+_{k+1}=1.$
It should be emphasized that the nature of the construction 
of $\Phi_{k+1}$ is combinatorial rather than geometric, in the sense that 
it heavily depends on the explicit presentations of the rings involved. 
In fact, they used a non-linear differential equation
called the {\it relativistic Toda lattice\/} introduced by Ruijsenaars. The map $\Phi_{k+1}$ arises as 
a solution of the relativistic Toda lattice equation
with its Lax matrix unipotent. For the affine side, the functions $h_i\;(1\leq i\leq k)$
can be thought of as the coordinates of a certain abelian centralizer subgroup in $PSL_{k+1}(\C).$}

There is a map $\sh: S_{k+1}\rightarrow \Par^k$ explicitly 
described by Lam and Shimozono \cite{LS:MRL}.
For $w\in S_{k+1}$, define 
$\Des(w)=\{1\leq i\leq k\;|\;w(i)>w(i+1)\}$.
The next result is a generalization of Theorem 1 in \cite{LS:MRL} by Lam and Shimozono 
for the (co)homology case to the $K$-theoretic setting. 
\begin{thm}\label{thm:KPet}
For $w\in S_{k+1}$, we have
\begin{equation*}
\Phi_{k+1}(\O^w_{\Fl})=\frac{\F^{-1}(\tilde{g}_{\sh(w)})
}{\prod_{i\in \Des(w)}\tau_i}.
\end{equation*}
\end{thm}
As noted above, $\F^{-1}(\tilde{g}_{\sh(w)})=\tfg{k}_{\sh(w)}$ is identified 
with the structure sheaf $\O_{\sh(w)}^\Gr$, while 
 $\tau_i$ is identified with the Schubert structure sheaf associated to an anti-dominant 
translation element in $\Waf$ (see \S\ref{sec:Khom} for a more precise statement).
Note also that for $w\in S_{k+1}$, we can take
the {\it quantum Grothendieck polynomials\/} $\mathfrak{G}_w^Q
\in \C[Q][z_1,\ldots,z_k]$ of Lenart and Maeno \cite{LM} (with the change of variables $x_i=1-z_i$)
as a representative for $\O^w_{\Fl}$ (see \S\ref{sec:QK}). 
The statement of 
Theorem \ref{thm:KPet} is a refinement of Conjecture 1.8 in \cite{IIM}.
See also Conjecture 2.12 (a), (b) in \cite{BMS}.

\subsection*{Future works and related results}
Now we are able to say that the unipotent solution of the relativistic Toda lattice  
actually gives the canonical $K$-theoretic Peterson isomorphism, in type $A$. 
A natural question is how to generalize 
this fact to any semisimple algebraic group $G$.
Furthermore, 
we would like to describe the $K$-theoretic 
Peterson isomorphisms in various types at the level of concrete polynomial representatives of Schubert classes. 
This leads us to a lot of interesting combinatorial problems related to the geometry
of $K$-theoretic Gromov-Witten theory. 
We know from Kim \cite{Kim} that the quantum cohomology 
ring of $G/B$ is identified with the quotient ring of 
a polynomial ring by the ideal generated by the conserved quantities 
of the Toda lattice for the Langlands dual group $G^\vee$ with the nilpotent initial condition.  
For our purpose, one of the central tasks is to obtain an analogue of Kim's result 
in the context of quantum $K$-theory ring for $G/B$.
Another possible clue in this direction is the work of Bezrukavnikov, Finkelberg, and Mirkovi\'c \cite{BFM}
that studied a connection between the $K$-homology of the 
affine Grassmannian and certain generalized Toda lattice equations. However, 
we still do not understand how this work fits into the framework 
of the $K$-Peterson isomorphism. 

For the affine Grassmannian side,  Lam, Schilling, and Shimozono \cite{LSS(C)} defined the $k$-Schur functions for the symplectic groups. It is remarkable that Seelinger \cite{See} made a conjecture
that the symplectic $k$-Schur functions of \cite{LSS(C)} can be 
expressed by a raising operator formula.              
Pon \cite{Pon} studied {the case of the orthogonal group. In particular, he gave a definition of ``$k$-Schur functions'' as 
the Schubert representatives for the homology of the affine Grassmannian in type $B$, 
and established the Pieri rule for types $B$ and $D$.}
Further combinatorial researches on these functions are needed to explore the issue
of giving explicitly the $K$-theoretic Peterson isomorphism.

Even in type $A$, there are a lot of things to do. 
We have a Chevalley type formula for $QK(\Fl)$ (\cite{LNS}, \cite{LP}, \cite{LM}, \cite{NOS}), 
which describes 
the multiplication by $\O^{s_i}_{\Fl}\;(1\leq i\leq k)$,
where $s_i$ is the transposition $(i,i+1)$.  
The image of  $\O^{s_i}_{\Fl}\;(1\leq i\leq k)$ under $\Phi_{k+1}$
is 
$\F^{-1}(\gt{k}_{R_i^*})/\tau_i$, where 
$R_i^*$ is the partition obtained from $R_i$ by removing the unique corner  
box.  We expect a good combinatorial formula for 
$\gt{k}_{R_i^*}\cdot \gt{k}_\lambda$. Although we know a 
dictionary between quantum and affine Schubert classes, the problem of 
translating the Chevalley formula into its affine counterpart 
seems not to be so simple. 
For the homology case, the analogous issue was pursued by Dalal and Morse 
 (\cite[Conjecture 39]{DM}). 
 Another basic question is what formula on  
 the quantum side is corresponding to $\tilde{g}^{(k)}_{(i)}\cdot \gt{k}_\lambda$ and $\tilde{g}^{(k)}_{(1^i)}\cdot \gt{k}_\lambda
\;(1<i<k-1).$
 Furthermore, there is a conjecture\footnote{After the first version of this paper was submitted the conjecture is proved in \cite{NS}} by Lenart and Maeno of a Pieri type formula
 for $\mathfrak{G}_w^Q$ (\cite[Conjecture 6.7]{LM}).
 We hope that the ``quantum equals affine'' phenomenon in $K$-theory would shed light on these questions. 
Also, it is natural to extend Theorem 1.3 torus equivariantly. In fact, Lam and Shimozono \cite{LM:equivToda} proved the equivariant (co)homology version of the explicit Peterson isomorphism in \cite{LS:MRL}. It was shown that the double $k$-Schur functions of Lam-Shimozono (\cite{LM:k-double}), and the quantum double Schubert polynomials by Kirillov and Maeno \cite{KM1} and by Ciocan-Fontanine and Fulton \cite{CF}, can be obtained from each other by the map. There are arguments in \cite[\S4]{LM:equivToda} on the centralizer family for $\SL_{k+1}$ in connection with Peterson's $j$-map, which would be useful in future studies.

\subsection*{Organization}
In Sections \ref{sec:defn}, we present some basic definitions and preliminary facts.
 In Section \ref{sec:proof}, we prove Theorem \ref{thm:main}.
 In Section \ref{sec:QK}, we discuss 
 the $K$-theoretic Peterson isomorphism. 
 In Appendix \ref{app:A}, we provide some results related to 
 the parabolic quotient of a Coxeter group, which are used in the proof of 
 Theorem \ref{thm:main}. 
 In Appendix \ref{sec:Grass}, we give a proof of a result (Lemma \ref{lem:lambdaGr}) on Grassmannian 
 permutations. In Appendix \ref{app:C}, we record the vertical Pieri rule for 
 the closed $K$-$k$-Schur functions. 
\subsection*{Acknowledgements} We thank Hiroshi Naruse, Daisuke Sagaki, and Mark Shimozono for 
valuable discussions; in particular, S.~N. thanks Daisuke Sagaki for helpful discussions about the material in Appendix A.
We thank Cristian Lenart for valuable comments.  
We also thank the anonymous referee for
valuable suggestions for improving 
the manuscript. 
T.~I. was partly supported  by JSPS Grant-in-Aid for Scientific Research 20H00119, 20K03571, 18K03261, 17H02838.
S.~N. was partly supported by JSPS Grant-in-Aid for Scientific Research (C) 21K03198. % Naito
S.~I. was partly supported by JSPS Grant-in-Aid for Scientific Research (C) 19K03605.

\section{Basic definitions}\label{sec:defn}
Let $k$ be a positive integer. In this section, we fix the notation and explain the definitions and some properties 
of basic notions needed to understand Theorem \ref{thm:main}.
\subsection{Affine symmetric groups}\label{sec:aff} 
The {\it affine symmetric group\/} $\Waf$ is the group with generators 
$\{s_i\;|\; i\in I\}$ for $I=\{0,1,\ldots,k\}$ subject to the relations:
$$s_i^2=id, \quad s_is_{i+1}s_i=s_{i+1}s_i s_{i+1},\quad
s_is_j=s_js_i\quad
\mbox{for} \quad
i-j\neq 0,\pm1,$$ with indices 
considered modulo $k+1$.
The {\it length} $\ell(w)$ of $w\in \Waf$ is the minimum number $m$ such that 
$w=s_{i_1}\cdots s_{i_m}$ for some $i_j\in I$; any such 
expression for $w$ with $\ell(w)$ generators is said to be {\it reduced}.
The set of {\it affine Grassmannian elements} $\Waf^0$ is the minimal length coset representatives 
for $\Waf/S_{k+1}$, where $S_{k+1}=\langle s_1,\ldots,s_k\rangle.$ 

The {\it Bruhat order\/} (or {\it strong order\/}) on $\Waf$ is denoted by $\leq$. 
It can be described by the subword property (see \cite[\S2]{BB}). 

\subsection{$k$-bounded partitions and affine Grassmannian elements}
Let $\Par^k_\ell:=\{(\lambda_1,\ldots,\lambda_\ell)\in \Z^\ell\;|\;
k\geq \lambda_1\geq \cdots \geq \lambda_\ell\geq 0
\}$ denote the set of partitions contained in the $\ell\times k$ rectangle and let $\Par^k$
be the set of partitions $\lambda$ with $\lambda_1\leq k$.
For a partition $\lambda$, the {\it length\/} $\ell(\lambda)$ is the
number of nonzero parts of $\lambda.$
There is a bijection $\Par^k\rightarrow \Wafo\;(\lambda \mapsto x_\lambda)$ due to Lapointe and Morse \cite[Definition 45, Corollary 48]{LM05}
given by 
\begin{equation}
x_\lambda:=
(s_{\lambda_{\ell}-\ell}\cdots s_{-\ell+1})\cdots
(s_{\lambda_2-2}\cdots s_{-1})
(s_{\lambda_1-1}\cdots s_0),\label{eq:redexp}
\end{equation}
where $\ell=\ell(\lambda).$
To make the notation simple, we are omitting the dependency of $x_\lambda$ on $k$.
\begin{example}
Let $k=4, \lambda=(4,3,2)$. The corresponding 
affine Grassmannian element in $\tilde{S}_{5}^0$ is
$x_\lambda=s_4s_3\cdot s_1s_0s_4\cdot s_3s_2s_1s_0$.
 This is obtained by reading the $(k+1)$-residues in each row of 
 $\lambda$, from right to left, proceeding with bottom row to top. 
$$
\ytableausetup{smalltableaux}
\ytableausetup{centertableaux}
\begin{ytableau}
0 & 1 & 2 &3  \\
4 & 0 &1 \\
3 & 4 
\end{ytableau}.
$$
\end{example}

For $\lambda,\mu\in \Par^k$, 
we denote $\lambda\leqk \mu$ if
$x_\lambda\leq x_\mu$ holds, where 
$\leq $ is the Bruhat order on $\Waf$. 
The following fundamental fact is included in the proof of 
\cite[Proposition 2.16]{BMS}.
\begin{lem}\label{lem:funny}
Suppose $\lambda\in \Par^k$
satisfies $\lambda_1+\ell(\lambda) \leq k+1.$
For $\mu\in \Par^k$,
$\mu\leqk \lambda\Longleftrightarrow \mu\subset \lambda.$
\end{lem}
\begin{rem}
For $\lambda,\mu\in \Par^k$, $\lambda\subset \mu$ implies
$\lambda\leqk \mu$. The reverse implication is not true in general.
For example, in $\tilde{S}_{3}$ we have
$(2,2)\underset{2}{\leq} (2,1,1,1)$ because the reduced expression 
$x_{(2,1,1,1)}=s_0s_1s_2s_1s_0$ has 
a subexpression $s_0s_2s_1s_0=x_{(2,2)}.$
\end{rem}

\subsection{Cyclically increasing elements in $\Wafo$}

  Let $A$ be a proper subset of $I=\{0,1,\ldots,k\}.$
  Set $|A|=r$.
  There is a sequence $(i_1,\ldots,{i_r})$
  consisting of the elements of $A$ such that
 an index $i+1$ never occurs anywhere to the left 
of an index $i\;(\mathrm{mod}\; k+1).$
For any such sequence, the element $s_{i_1}\cdots s_{i_r}\in\Waf$ depends only on the set of 
indices $A=\{i_1,\ldots,i_r\}$ (see also Remark \ref{rem:fulcom}).
Let us denote the element by $u_A.$ Such an  element in $\Waf$ is called 
a {\it cyclically increasing element.}
Similarly, for any $A\subsetneq I$, we can define
the corresponding {\it cyclically decreasing element\/}
denoted by $d_A.$
We choose a sequence $(i_1,\ldots,i_r)$
such that an index $i$ never occurs anywhere to the right
of an index ${i+1}\;(\mathrm{mod}\; k+1)$, and
set $d_A:=s_{i_1}\cdots s_{i_r}.$

 \begin{example}
  Let $k=4$.
  Then  for $A=\{0,1,3,4\}$, $u_{A}=s_3 s_4s_0s_1,\;
  d_{A}=s_1 s_0s_4s_3$, 
  and for $A=\{0,2,4\}$, $u_{A}=s_4s_0s_2
  =s_4s_2s_0=s_2s_4s_0,\;
  d_{A}=s_2s_0s_4
  =s_0s_2s_4=s_0s_4s_2.$ 
  \end{example}

\subsection{Dual stable Grothendieck polynomials}\label{sec:dualstable}

We work in the ring 
$\Lambda=\Z[e_1,e_2,\ldots]
=\Z[h_1,h_2,\ldots]$ of symmetric functions in infinitely many variables $(x_1,x_2,\ldots)$,
where $e_r=\sum_{i_1<\cdots<i_r}x_{i_1}\cdots x_{i_r}$ and
$h_r=\sum_{i_1\leq \cdots\leq i_r}x_{i_1}\cdots x_{i_r}$.
Set $h_0=e_0=1$ and $h_r=0$ for $r<0$ by convention.

Let $\F$ be the ring automorphism of $\Lambda$ defined by $\F(h_i)=\tilde{h}_i\;(i\geq 1)$,
where we set $\tilde{h}_i=\sum_{j=0}^ih_j$.
Note that $\F^{-1}$ sends $h_i$ to $h_i-h_{i-1}$ for $i\geq 1.$
For $i,m\in \Z$, define
\begin{equation*}
h_i^{(m)}:=\F^{m}(h_i)=
\sum_{j=0}^i\binom{m+j-1}{j}h_{i-j},
\end{equation*}
where $\binom{n}{i}=n(n-1)\cdots(n-i+1)/i!$ for $n\in \Z,\;i\geq 1$ and
$\binom{n}{0}=1$ for $n\in \Z$. 
For $\gamma\in \Z^\ell$,
define
\begin{equation}
g_{\gamma}:=\det(h_{\gamma_i+j-i}^{(i-j)})_{1\leq i,j\leq \ell}.\label{eq:g}
\end{equation}
When $\lambda$ is a partition,
$g_\lambda$ is the {\it dual stable Grothendieck polynomials\/}.
Although the defining formula for $g_\gamma$ in \cite{BMS} is $\det(h_{\gamma_i+j-i}^{(i-1)})_{1\leq i,j\leq \ell}$, by some column operations, we see that 
their definition agrees with \eqref{eq:g}. 
Note that $h_i^{(m)}$ is denoted by $k_i^{(m)}$ in \cite{BMS}. 

The following result is fundamental and used throughout the paper.  
\begin{prop}[Takiguku \cite{Tak1}]\label{prop:Taki-g} For any partition $\lambda$,
set $\tilde{g}_\lambda:=\sum_{\mu\subset \lambda}g_\mu$. Then we have
\begin{equation}
\F(g_\lambda)=\tilde{g}_\lambda.\label{eq:Taksum}
\end{equation}
\end{prop}

\subsection{$K$-$k$-Schur functions}\label{sec:Kk}
For $i\in I$, 
set for  $w\in \Waf$, 
\begin{equation}
s_i*w=\begin{cases} s_iw &(s_iw>w)\\
w & (s_iw<w)
\end{cases}.\label{eq:*}
\end{equation}If we write $\phi_{i}: \Waf\rightarrow \Waf$
$(w\mapsto s_i*w)$, then
$\phi_{i}^2=\phi_i, \;\phi_{i}\phi_{i+1}\phi_{i}=
\phi_{i+1}\phi_i \phi_{i+1}$ (see the proof of \cite[Proposition 2.1]{Ste07}). 
So for $v\in \Waf$,
we can define  
%the {\it Demazure action\/} 
%of $v$ by 
\begin{equation*}
v*w=s_{i_1}*(s_{i_2}*\cdots (s_{i_r}*w)\cdots)\quad (w\in\Waf),
\end{equation*}
where $v=s_{i_1}\cdots s_{s_r}$ is an arbitrary 
 reduced expression. 
 
For $x\in \Wafo$, we write $g_x^{(k)}$ for $g^{(k)}_\lambda$ with $x=x_\lambda\,,\lambda\in \Par^k.$
 
\begin{defn}\label{def:KkSchur}
The $K$-$k$-Schur functions $\{{g}^{(k)}_\lambda\}_{\lambda\in \Par^k}$ are the family of elements of $\Lambda_{(k)}$ such that $g^{(k)}_{\varnothing}=1$ and 
\begin{equation}
{h}_{r}\cdot {g}^{(k)}_\lambda
=
\sum_{\substack{A\subset I,\;|A|= r\\
d_A*x_\lambda\in \Wafo}}
(-1)^{|A|-\ell(d_A*x_\lambda)+\ell(x_\lambda)}
{g}_{d_A*x_\lambda}^{(k)}\label{eq:11}
\end{equation}
for $\lambda\in \Par^k$ and $1\leq r\leq k$.
\end{defn}

\begin{example} Let $k=2,\; \lambda=(1,1,1),$ and $r=2$.
Then $x_{\lambda}=s_1s_2s_0$. 
There are three $A$'s: $\{1,0\},\{0,2\},\{1,2\}.$
We compute 
\begin{eqnarray*}
d_{\{1,0\}}*x_{\lambda}
&=&(s_1s_0)*(s_1s_2s_0)=s_1s_0s_1s_2s_0=s_0s_1s_0s_2s_0
=s_0s_1s_2s_0s_2\notin \tilde{S}^{0}_{3},\\
d_{\{0,2\}}*x_{\lambda}
&=&(s_0s_2)*(s_1s_2s_0)
=s_0s_2s_1s_2s_0=x_{(2,1,1,1)}\in\tilde{S}_3^{0},\\
d_{\{1,2\}}*x_{\lambda}
&=&(s_2s_1)*(s_1s_2s_0)
=s_2s_1s_2s_0=x_{(2,1,1)}\in\tilde{S}_3^{0},
\end{eqnarray*}
and hence
\begin{equation*}
h_{2}\cdot g_{(1,1,1)}^{(2)}=
g_{(2,1,1,1)}^{(2)}-g_{(2,1,1)}^{(2)}.
\end{equation*}
 \end{example}
 It is known that $\{g^{(k)}_\lambda\}_{\lambda\in \Par^k}$ is a basis of $\Lambda_{(k)}$ (\cite{LSS}). 
\begin{prop}[\cite{LSS}]
For $1\leq r\leq k$,
$g_{(r)}^{(k)}=h_r.$
\end{prop}

Recall that we set $\tilde{g}_\lambda:=\sum_{\mu\subset \lambda}g_\mu$ for any partition $\lambda$.
\begin{prop}\label{prop:ksmallgtilde}
For $\lambda\in \Par^k$ such that $\lambda_1+\ell(\lambda)\leq k+1$,
\begin{equation*}
\gt{k}_\lambda=
\tilde{g}_\lambda.
\end{equation*}
\end{prop}
\begin{proof}
From 
Lemma \ref{lem:funny}, and  \eqref{eq:M12}, we see that
\begin{equation*}
\gt{k}_\lambda
=\sum_{\mu\leqk\lambda}
g^{(k)}_\mu
=\sum_{\mu\subset\lambda}
g^{(k)}_\mu
=\sum_{\mu\subset\lambda}
g_\mu
=\tilde{g}_\lambda.
\end{equation*}
\end{proof}

\subsection{Katalan functions}
Fix a positive integer $\ell$. 
Let $\{\ep_1,\ldots,\ep_\ell\}$ be the standard basis of $\Z^\ell.$  
By a {\it positive root\/} $\beta$, we mean an element
of the form $\ep_i-\ep_j\in \Z^\ell$ with $1\leq i<j\leq \ell$, which is also
denoted by $(i,j)$. The set of all positive roots is denoted by 
$\Delta_{\ell}^+$. Although this is considered as the set of positive roots of type $A_{\ell-1}$, we use this notation $\Delta_\ell^+$
following \cite{BMS} 
rather than $\Delta_{\ell-1}^+.$

A natural partial order $\leq $ on $\Delta_\ell^+$
is defined by $\alpha\leq \beta$ if 
$\beta-\alpha$ is a linear combination of 
positive roots with coefficients in $\Z_{\geq 0}.$ 
An upper order ideal $\Psi$ of $\Delta_{\ell}^+$
is called a {\it root ideal\/}.

Given a root ideal $\Psi\subset \Delta^+_\ell$, a multiset $M$ supported on $\{1,\ldots,\ell\}$,
and $\gamma\in \Z^\ell$,
we call $(\Psi,M,\gamma)$ a {\it Katalan triple\/}.
Let $m_M: \{1,\ldots,\ell\}\rightarrow \Z_{\geq 0}$ 
denote the multiplicity function of $M.$ 
Each Katalan triple $(\Psi,M,\gamma)$ can be depicted by an
$\ell\times \ell$ grid of square boxes (labeled by matrix-style coordinates) with the boxes of $\Psi$ shaded,
$m_M(a)$ $\bullet$'s in column $a$ (assuming $m_M(a)<a$), and the entries of $\gamma$ written along the diagonal boxes. 
\begin{example}
Let $\Psi=\{(1,2),(1,3),(1,4),(1,5),(2,4),(2,5),(3,5)\}\subset \Delta_5^+$, 
$M=\{3,4,5,5\}$, and $\gamma=(3,2,0,1,0)$.
The Katalan triple $(\Psi,M,\gamma)$ is depicted by
$$\begin{ytableau}
{3} &*(lightgray)&*(lightgray) \bullet&*(lightgray)\bullet&*(lightgray)\bullet\\
{} &2&&*(lightgray)&*(lightgray)\bullet\\
{} &&0&&*(lightgray)\\
{} &&&1&\\
{} &&&&0\\
\end{ytableau}.$$
\end{example}  

We define the {\it Katalan function\/} 
associated to the triple $(\Psi,M,\gamma)$ by
\begin{eqnarray*}
K(\Psi;M;\gamma)&:=&
\prod_{j\in M}(1-L_j)
\prod_{(i,j)\in \Delta_\ell^+\setminus \Psi}
(1-R_{ij})k_\gamma,\\
k_\gamma&:=&h_{\gamma_1}^{(0)}h_{\gamma_2}^{(1)}\cdots h_{\gamma_\ell}^{(\ell-1)}.
\end{eqnarray*}
Note that raising operators are not well-defined as linear transformations on $\Lambda$.
They act on the subscript $\gamma$ of $k_\gamma$ rather than the function $k_\gamma$.   
A rigorous formulation can be found in \cite[\S3]{BMS}.
For any root ideal $\mathcal{L}\subset \Delta^+_\ell$,
let 
\begin{equation*}
L(\mathcal{L})=\bigsqcup_{(i,j)\in \mathcal{L}}\{j\}.
\end{equation*}
We also write $K(\Psi;L(\mathcal{L});\gamma)$ simply 
 as $K(\Psi;\mathcal{L};\gamma).$

In \cite{BMS}, Blasiak, Morse, and Seelinger introduced, for $\lambda\in \Par^k_\ell$,
\begin{eqnarray*}
\fg{k}_\lambda&:=&K(\Delta^k(\lambda);\Delta^{k+1}(\lambda);\lambda),\\
\tfg{k}_\lambda&:=&K(\Delta^k(\lambda);\Delta^k(\lambda);\lambda).
\end{eqnarray*}
$\fg{k}_\lambda$ is called a $K$-{\it Schur Katalan function\/},
and 
$\tfg{k}_\lambda$ a {\it closed\/} $K$-{\it Schur Katalan function.}
If we choose $\ell$ such that $\ell\geq \ell(\lambda),$ then $\fg{k}_\lambda$ and $\tfg{k}_\lambda$ do not depend on $\ell$ (\cite[Lemma 3.4, Remark 3.5]{BMS}).

The following simplified formula is available.  
\begin{prop}
Let $\lambda\in \Par^k_\ell,$ then 
\begin{equation}
\tfg{k}_\lambda=\prod_{(i,j)\in \Delta_\ell^+\setminus \Delta^k(\lambda)}
(1-L_j)^{-1}
(1-R_{ij})h_\lambda,\label{eq:LRKH}
\end{equation}
where $h_\lambda=h_{\lambda_1}\cdots h_{\lambda_\ell}$, and 
$L_j,R_{ij}$ act on the subscript $\lambda$.
\end{prop}
\begin{proof}
In the proof of \cite[Proposition 2.3]{BMS}, it was shown that 
$\prod_{(i,j)\in \Delta_\ell^+}(1-L_j)k_\gamma=h_\gamma.$
\eqref{eq:LRKH} follows from this.
\end{proof}

\begin{example}
For $k=3$, $\lambda=(2,1,1)$, $\tfg{3}_\lambda$ is depicted by
\begin{ytableau}
2 & & *(lightgray) \bullet\\
& 1 & \\
&& 1
\end{ytableau}. We have
\begin{eqnarray*}
\tfg{3}_{(2,1,1)}&=&
(1-L_2)^{-1}(1-L_3)^{-1}
(1-R_{12})(1-R_{23}) h_{211}\\
&=&
{h_{211}+h_{201}-h_{220}-h_{301}+h_{310}}\\
&=&
{h_1^2h_2+h_1h_2-h_2^2}\\
&=&{h_2(h_1^2+h_1-h_2).}
\end{eqnarray*}
\end{example}
One of the main results in \cite{BMS} is
\begin{equation*}
\fg{k}_\lambda=g^{(k)}_\lambda.
\end{equation*}
The main result of the present paper (Theorem \ref{thm:main}) 
is
\begin{equation}
\tfg{k}_\lambda=\F^{-1}(\gt{k}_\lambda).\label{eq:mainmain}
\end{equation}

A simple consequence of \eqref{eq:mainmain} is the following. 
\begin{cor}[{\cite[Proposition 2.16 (d)]{BMS}}]

If $\lambda_1+\ell(\lambda)\leq k+1$, then 
$\tfg{k}_\lambda=g_\lambda.$
\end{cor}
\begin{proof} By Lemma \ref{prop:ksmallgtilde} and Proposition \ref{prop:Taki-g},
$$
\tfg{k}_\lambda=\F^{-1}(\gt{k}_\lambda)=
\F^{-1}(\tilde{g}_\lambda)=g_\lambda.
$$
\end{proof}
\section{Proof of Theorem \ref{thm:main}}\label{sec:proof}
Before we start the proof of 
Theorem \ref{thm:main}, we gather some results on 
$K$-$k$-Schur functions and Katalan functions in the first three subsections. In \S\ref{sec:kconj}, we explain
some results on the $k$-conjugation. In \S\ref{sec:basicK}, we collect some basic 
properties of Katalan functions used in \S\ref{sec:PfKey}. In \S\ref{sec:zeroH}, we introduce an action of 
the $0$-Hecke algebras on $\Lambda_{(k)}$.
The outline of the proof of
Theorem \ref{thm:main} is given in \S\ref{sec:outline}.
The last subsection is devoted to the proof of 
a key lemma (Lemma \ref{lem:main}).
With the help of a general fact (Lemma \ref{lem:Naito2++}) on 
the parabolic coset space of a Coxeter group,
we complete the proof of Theorem \ref{thm:main}.
\subsection{$k$-conjugation}\label{sec:kconj}
There is an automorphism $\omega_k$ of $\Waf$
given by 
$\inv_k(s_i)=s_{-i}=s_{k+1-i}$ for $i\in I$.
Note that $\inv_k$ fixes $s_0$.
In fact, $\omega_k$ is an automorphism of a Coxeter group. 
So it is easy to see 
\begin{eqnarray}
w\leq v &\Longleftrightarrow&
\omega_k(w)\leq \omega_k(v).\label{eq:omega_strong}
\end{eqnarray}

The left weak order $\leq _L$ on $\Waf$ is defined by the covering relation
$$
w\lessdot_L v
\Longleftrightarrow
v=s_iw \;\mbox{and}\;
\ell(v)=\ell(w)+1\;\mbox{for some}\;
i\in I.
$$
It is easy to see that
$w\lessdot_L v\Longleftrightarrow \omega_k(w)\lessdot_L\omega_k(v)$, 
and hence we have 
\begin{eqnarray}
w\leq _Lv &\Longleftrightarrow&
\omega_k(w)\leq _L \omega_k(v).\label{eq:omega_weak}
\end{eqnarray}
For $A=\{i_1,\ldots,i_m\}\subsetneq I$, we have 
\begin{equation*}
\inv_k(u_A)=d_{\overline{A}},\quad
\inv_k(d_A)=u_{\overline{A}},
\end{equation*}
where 
$\overline{A}:=\{-i_1,\ldots,-i_m\}.$

\begin{defn} Let $\Omega$ be the ring morphism of $\Lambda$ defined by 
\begin{equation}
\Omega(h_i)=g_{(1^i)}\quad (i\geq 1).
\end{equation}
\end{defn}

\begin{prop}
$\Omega$ is an involution on $\Lambda$ and  
$\Omega$ commutes with $\F$.
\end{prop}
\begin{proof}
A proof of the fact that $\Omega$ is an involution 
can be found in \cite[\S8]{M12}.
The commutativity follows from \eqref{eq:Taksum}:
\[
\Omega (\F(h_i))=\Omega(\sum_{j=0}^ih_j)
=\sum_{j=0}^ig_{(1^j)}
=\F(g_{(1^i)})=\F(\Omega(h_i)).
\]
\end{proof}

It is easy to see that $\omega_k$ preserves $\Wafo$. Hence for $\lambda\in \Par^k$, 
$\omega_k(x_\lambda)=x_{\mu}$ for some $\mu\in \Par^k.$ 
Then we define $\omega_k(\lambda)=\mu$.
An explicit description of $\omega_k(\lambda)$, also denoted by $\lambda^{\omega_k}$ in \cite{book}, \cite{M12}, is available (see \cite[\S1.3]{book}).

\begin{thm}[\cite{M12}]\label{thm:omegag}
For $x\in\Wafo$, 
$\Omega(g_x^{(k)})=g_{\inv_k(x)}^{(k)}.$
Equivalently,
for $\lambda\in \Par^k,$ 
$\Omega(g_{\lambda}^{(k)})=g_{\omega_k(\lambda)}^{(k)}$.
\end{thm}

\begin{cor}\label{cor:omegatilde}
 For $\lambda\in \Par^k$, 
$\Omega(\gt{k}_\lambda)=\gt{k}_{\omega_k(\lambda)}$.
\end{cor}
\begin{proof}
For $x\in \Wafo$, $\gt{k}_x=\sum_{y\leq x}g_{y}^{(k)}$, where
$y$ are the elements in $\Wafo$ such that $y\leq x$ in the Bruhat order. From 
Theorem \ref{thm:omegag}, and the fact 
\eqref{eq:omega_strong}, we see that
\begin{equation*}
\Omega(\gt{k}_x)=
\Omega(\sum_{y\leq x}g_{y}^{(k)})
=\sum_{y\leq x}\Omega(g_{y}^{(k)})
=\sum_{y\leq x}g_{\inv_k(y)}^{(k)}
=\sum_{\inv_k(y)\leq \inv_k(x)}g_{\inv_k(y)}^{(k)}
=\tilde{g}^{(k)}_{\inv_k(x)},
\end{equation*}
where the last equality holds since $\omega_k$ is an involution.
Set $x=x_\lambda$. Then  $\Omega(\gt{k}_\lambda)=\gt{k}_{\omega_k(\lambda)}$
\end{proof}

\subsection{Basic properties of Katalan functions}\label{sec:basicK}
Let $\Psi\subset \Delta_\ell^+$ be a root ideal.
A root $\alpha\in \Psi$ is a {\it removable root\/} of $\Psi$ if $\Psi\setminus \alpha$ is a root ideal. A root $\beta\in \Delta_\ell^+$ is 
an {\it addable root\/} of $\Psi$ if $\Psi\cup \alpha$ is a 
root ideal. 

We define an oriented graph with 
$\{1,\ldots,\ell\}$ as the vertex set and  
the oriented edges $j\rightarrow i$ if $(i,j)$ 
is a removable root in $\Psi$ (in \cite{BMS} the bounce graph is not considered as an oriented graph, but here it is).
An edge of the bounce graph of $\Psi$ is called simply  
a {\it bounce edge\/} of $\Psi$. 
Let $p\in \{1,\ldots,\ell\}$. If there is a bounce edge $p\rightarrow i$ of  $\Psi$, then such $i$ is unique by the construction, and we write $i=\up_\Psi(p).$

\begin{figure}[htbp]
\centering
% \begin{picture}(100,110)
% \multiput(10,10)(0,10){10}{\line(1,0){90}}
% \multiput(10,10)(10,0){10}{\line(0,1){90}}
% \multiput(12.5,92.5)(10,-10){9}{$\bullet$}
% \linethickness{0.2mm}
% \put(35,95){\vector(-1,0){17.7}}
% \put(35,75){\line(0,1){20}}
% %
% \put(55,85){\vector(-1,0){27.7}}
% \put(55,55){\line(0,1){30}}
% \put(75,55){\vector(-1,0){17.7}}
% \put(75,35){\line(0,1){20}}
% %
% \put(65,65){\vector(-1,0){17.7}}
% \put(65,45){\line(0,1){20}}
% \put(95,45){\vector(-1,0){27.7}}
% \put(95,15){\line(0,1){30}}
% \linethickness{0.01mm}
% \multiput(30,90)(0,0.5){20}{\line(1,0){70}}
% \multiput(50,80)(0,0.5){20}{\line(1,0){50}}
% \multiput(60,60)(0,0.5){40}{\line(1,0){40}}
% \multiput(70,50)(0,0.5){20}{\line(1,0){30}}
% \multiput(90,40)(0,0.5){20}{\line(1,0){10}}
% \end{picture}
% \qquad
{
\begin{tikzpicture}[scale = .4]
\filldraw [lightgray] (2,9)--(9,9)--(9,3)--(8,3)--(8,4)--(6,4)--(6,5)--(5,5)--(5,7)--(4,7)--(4,8)--(2,8)--cycle;
\foreach \i in {0,1,...,8}{
   \coordinate (\i) at ({\i + .5},{8.5 - \i});
   \draw[fill] (\i) circle (.15);
}
\foreach \i in {0,1,...,9}{
   \draw (0,\i)--(9,\i);
   \draw (\i,0)--(\i,9);
}
\draw [-stealth,thick] (8) -- ++(0,3) -- ++(-2.7,0);
\draw [-stealth,thick] (6) -- ++(0,2) -- ++(-1.7,0);
\draw [-stealth,thick] (5) -- ++(0,2) -- ++(-1.7,0);
\draw [-stealth,thick] (4) -- ++(0,3) -- ++(-2.7,0);
\draw [-stealth,thick] (2) -- ++(0,2) -- ++(-1.7,0);
\end{tikzpicture}
}
%\begin{picture}(100,110)
%\multiput(10,10)(0,10){10}{\line(1,0){90}}
%\multiput(10,10)(10,0){10}{\line(0,1){90}}
%\multiput(12.5,92.5)(10,-10){9}{$\bullet$}
%\linethickness{0.2mm}
%\put(35,95){\vector(-1,0){17.7}}
%\put(35,75){\line(0,1){20}}
%%
%\put(55,85){\vector(-1,0){27.7}}
%\put(55,55){\line(0,1){30}}
%\put(75,55){\vector(-1,0){17.7}}
%\put(75,35){\line(0,1){20}}
%%
%\put(65,65){\vector(-1,0){17.7}}
%\put(65,45){\line(0,1){20}}
%\put(95,45){\vector(-1,0){27.7}}
%\put(95,15){\line(0,1){30}}
%\linethickness{0.01mm}
%\multiput(30,90)(0,0.5){20}{\line(1,0){70}}
%\multiput(50,80)(0,0.5){20}{\line(1,0){50}}
%\multiput(60,60)(0,0.5){40}{\line(1,0){40}}
%\multiput(70,50)(0,0.5){20}{\line(1,0){30}}
%\multiput(90,40)(0,0.5){20}{\line(1,0){10}}
%\end{picture}
\caption{Bounce graph of $\Psi$}
\label{fig:path}
\end{figure}
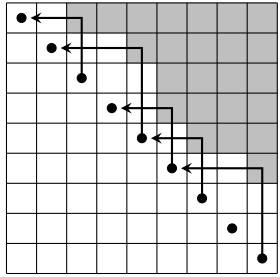
Each connected component of the bounce graph of $\Psi$ is called a {\it bounce path\/} of $\Psi$. 
For $p\in \{1,\dots,\ell\}$, $\top_\Psi(p)$ denotes the smallest element in the bounce path containing $p$ (see \cite{BMS}). 
%Note in \cite{BMS}, if there is no $i$ such that $p\rightarrow i$, then $\top_{\Psi}(p)$ is not defined, but in our definition $\top_{\Psi}(p)=p$. 
\begin{example}
For the root ideal $\Psi$ in Figure 1, $\{1,3\},\;\{2,5,7\},\;\{4,6,9\},\{8\}$
are the bounce paths. We have, for example,
$
\up_\Psi(3)=1
$, 
$\up_\Psi(7)=5$, 
$
\top_\Psi(7)=\top_\Psi(5)=\top_\Psi(2)=2.$
\end{example}

\begin{defn}[Walls and ceilings]
Let $d$ be a positive integer. 
A root ideal $\Psi$ is said to have a {\it wall\/} in rows $r,r+1,\ldots,r+d$ if 
the rows $r,r+1,\ldots,r+d$ of $\Psi$
have the same length, and 
a {\it ceiling\/} in 
columns $c,c+1,\ldots,c+d$ if the columns $c,c+1,\ldots,c+d$ of $\Psi$ have the
same length.
\end{defn}
\begin{example}
The root ideal $\Psi$ in Figure 1 
has a ceiling in columns $1,2$, in columns $3,4$, in columns $7,8$, and has a wall 
in rows $3,4$, in rows $7,8,9$. 
\end{example}
Let us begin with an obvious remark.
\begin{lem}
Let $\Psi$ be a root ideal of $\Delta_\ell^+$, and $p\in \{2,\ldots,\ell\}$.
If $\top_\Psi(p)=p$, i.e., if there is no bounce edge starting from $p$, then 
$\Psi$ has has a ceiling in columns $p-1,p$. 
\end{lem}

The following lemmas are borrowed from \cite[Lemma 4.4]{BMS}. 
\begin{lem}[Adding or removing a root]\label{lem:addroot}
Let $(\Psi,M,\gamma)$ be a Katalan triple.

\begin{itemize}
\item[(i)] For any addable root $\alpha$ of $\Psi$,
\begin{equation*}
\K{\Psi}{M}{\gamma}
=\K{\Psi\cup \alpha}{M}{\gamma}
-\K{\Psi\cup \alpha}{M}{\gamma+\alpha}.
\end{equation*}
\item[(ii)] For any removable root $\alpha$ of $\Psi$,
\begin{equation*}
\K{\Psi}{M}{\gamma}
=\K{\Psi\setminus \alpha}{M}{\gamma}
+\K{\Psi}{M}{\gamma+\alpha}.
\end{equation*}
\end{itemize}
\end{lem}

\begin{example} We apply 
Lemma \ref{lem:addroot} (i) to the following Katalan triple,
with $\alpha=\ep_2-\ep_3$:  
$$
\begin{ytableau}
1 &  &  *(lightgray) &*(lightgray) \\
   & 0                && *(lightgray) \\
   & & 1 & *(lightgray) \\
   & &  & 1
\end{ytableau}=\begin{ytableau}
1 &  &  *(lightgray) &*(lightgray) \\
   & 0                &*(lightgray)& *(lightgray) \\
   & & 1 & *(lightgray) \\
   & &  & 1
\end{ytableau}
-\begin{ytableau}
1 & &  *(lightgray) &*(lightgray) \\
   & 1                &*(lightgray)& *(lightgray) \\
   & & 0 & *(lightgray) \\
   & &  & 1
\end{ytableau}.
$$We apply 
Lemma \ref{lem:addroot} (ii) to the following Katalan triple,
with $\alpha=\ep_3-\ep_4$:  
$$
\begin{ytableau}
1 &  &  *(lightgray) &*(lightgray) \\
   & 0                && *(lightgray) \\
   & & 1 & *(lightgray) \\
   & &  & 1
\end{ytableau}=\begin{ytableau}
1 &  &  *(lightgray) &*(lightgray) \\
   & 0                && *(lightgray) \\
   & & 1 &  \\
   & &  & 1
\end{ytableau}
+\begin{ytableau}
1 & &  *(lightgray) &*(lightgray) \\
   & 1                && *(lightgray) \\
   & & 2 &*(lightgray) \\
   & &  & 0
\end{ytableau}.
$$
\end{example}

\begin{lem}[Adding or removing a dot]\label{lem:dot}
Let $(\Psi,M,\gamma)$ be a Katalan triple.

\begin{itemize}
\item[(i)] For any $j\in M$,
$
\K{\Psi}{M}{\gamma}
=\K{\Psi}{M\setminus \{j\}}{\gamma}
-\K{\Psi}{M\setminus \{j\}}{\gamma-\ep_j},
$
\item[(ii)] For any $1\leq j\leq \ell$,
$\K{\Psi}{M}{\gamma}
=\K{\Psi}{M\sqcup \{j\}}{\gamma}
+\K{\Psi}{M}{\gamma-\ep_j}.$
\end{itemize}
\end{lem}

The next lemma is \cite[Lemma 3.3]{BMS}.

\begin{lem}[Alternating property]
Let $\Psi\subset \Delta^+_\ell$ be a root ideal, 
and $M$ a multiset on $\{1,\ldots,\ell\}$.
Suppose there is an index $1\leq i\leq\ell-1$ such that 
\begin{itemize}
\item[(a)] $\Psi$ has a ceiling in columns $i,i+1$,
\item[(b)] $\Psi$ has a wall in rows $i,i+1$, 
\item[(c)] $m_{M}(i+1)=m_M(i)+1$.
\end{itemize}
Then for any $\gamma\in \Z^\ell$,
\begin{equation}
\K{\Psi}{M}{\gamma}=-\K{\Psi}{M}{s_i\gamma-\ep_i+\ep_{i+1}}.
\label{eq:alt}
\end{equation}
In particular, if $\gamma_{i+1}=\gamma_{i}+1$ holds, then
 \begin{equation}
 \K{\Psi}{M}{\gamma}=0.\label{eq:vanish}
 \end{equation}
\end{lem}

\subsection{The $0$-Hecke algebra }\label{sec:zeroH}
The $0$-{\it Hecke algebra} $H_{k+1}$ is the associative $\C$-algebra generated by 
$\{T_i\;|\;i\in I\}$ subject to the relations : 
$$T_i^2=-T_i, \quad T_iT_{i+1}T_i=T_{i+1}T_i T_{i+1},\quad
T_iT_j=T_jT_i\quad
\mbox{for} \quad
i-j\neq 0,\pm1,$$
with indices 
considered modulo $k+1$.
 For $w\in \Waf$, define $T_w=T_{i_1}\cdots T_{i_m}$ for any reduced expression $w=s_{i_1}\cdots s_{i_m}.$
The elements $T_w\;(w\in \Waf)$ form a basis of $H_{k+1}.$

We introduce a family of symmetric functions
$\gcirc{k}_{\lambda}\;(\lambda
\in \Par^k)$ such that
$\tfg{k}_\lambda=\sum_{\mu\leqk \lambda}
\gcirc{k}_{\mu}
$
for all $\lambda\in \Par^k.$
Such functions uniquely exist since the transition matrix from
$\{\gcirc{k}_{\lambda}\}$ to 
$\{{\mathfrak{g}}_{\lambda}^{(k)}\}$ is upper uni triangular.

\begin{prop}\label{prop:Tgcirc}
There is a left $H_{k+1}$-module structure on $\Lambda_{(k)}$ such that 
\begin{equation*}
T_i\cdot \gcirc{k}_\lambda=\begin{cases}
\gcirc{k}_{s_i \lambda}& (s_ix_\lambda>x_\lambda\;\mbox{and}\;
s_i x_\lambda\in \Wafo)\\
-\gcirc{k}_\lambda & (s_ix_\lambda<x_\lambda)\\
0 & \mbox{(otherwise)}
\end{cases},
\end{equation*}
for $\lambda\in \Par^k$ and $i\in I$.
Moreover, we have
\begin{equation*}
D_i\cdot \tfg{k}_\lambda=\begin{cases}
\tfg{k}_{s_i \lambda}& (s_ix_\lambda>x_\lambda\;\mbox{and}\;
s_i x_\lambda\in \Wafo)\\
\tfg{k}_\lambda & \mbox{(otherwise)}
\end{cases},
\end{equation*}
for $i\in I$ and $\lambda\in \Par^k$, 
and $D_i:=T_i+1$.
\end{prop}

A proof of this proposition is given in \S\ref{sec:0Hmod}.
\begin{rem}
Note that we will eventually show $\F(\gcirc{k}_{\lambda})=g_\lambda^{(k)}.$
There is an action of $H_{k+1}$ on $\Lambda_{(k)}$ used in \cite{BMS}. If we denote the action of
$T_i$ in \cite[\S5.4]{BMS} by $T_i'$, then we have $T_i'=\F\circ T_i \circ \F^{-1}$.
\end{rem}

\subsection{Outline of Proof of Theorem \ref{thm:main}}\label{sec:outline}
The following characterization property
for the closed $K$-$k$-Schur functions is available. 
Note that for $1\leq r\leq k$, we have $\tilde{g}^{(k)}_{(1^r)}=\tilde{g}_{(1^r)}$ and 
$\tilde{g}^{(k)}_{(r)}=\tilde{g}_{(r)}=\tilde{h}_r$ (Proposition \ref{prop:ksmallgtilde}).    
\begin{lem}\label{lem:11+}
For $\lambda\in \Par^k$ and $1\leq r\leq k$, we have
\begin{equation}
\tilde{g}_{(1^r)}\cdot \tilde{g}^{(k)}_\lambda
=\sum_{\mu\leqk \lambda}
\sum_{\substack{A\subset I,\;|A|\leq r\\
u_A*x_\mu\in \Wafo}}
(-1)^{|A|-\ell(u_A*\mu)+\ell(\mu)}
{g}_{u_A*x_\mu}^{(k)}.\label{eq:11+}
\end{equation}
Moreover, the $\{\tilde{g}^{(k)}_\lambda\}_{\lambda\in \Par^k}$ are the unique elements of $\Lambda_{(k)}$ satisfying \eqref{eq:11+} for $1\leq r\leq k$.
\end{lem}
\begin{proof}
By summing up \eqref{eq:11} over $\mu\in \Par^k$ such that 
$\mu\leqk \lambda$ and integers $i$ with  $0\leq i\leq r$, we obtain 
\begin{equation}
\tilde{g}_{(r)}\cdot \tilde{g}^{(k)}_\lambda
=\sum_{\mu\leqk \lambda}
\sum_{\substack{A\subset I,\;|A|\leq r\\
d_A*x_\mu\in \Wafo}}
(-1)^{|A|-\ell(d_A*\mu)+\ell(\mu)}
{g}_{d_A*x_\mu}^{(k)}.\label{eq:Tak11+}
\end{equation}
In fact, this identity appears in \cite[p. 470]{Tak}.
We apply the $k$-conjugation to both sides of \eqref{eq:Tak11+}, and use Corollary \ref{cor:omegatilde}, to have
\begin{eqnarray*}
\tilde{g}_{(1^r)}\cdot \tilde{g}^{(k)}_{\omega_k(\lambda)}
&=&\sum_{\mu\leqk \lambda}
\sum_{\substack{A\subset I,\;|A|\leq r\\
d_A*x_\mu\in \Wafo}}
(-1)^{|A|-\ell(d_A*\mu)+\ell(\mu)}
{g}_{u_{\overline{A}}*x_{\omega_k(\mu)}}^{(k)}.
\end{eqnarray*}
This is equivalent to 
\begin{eqnarray}
\tilde{g}_{(1^r)}\cdot \tilde{g}^{(k)}_{\lambda}
&=&\sum_{\omega_k(\mu)\leqk \omega_k(\lambda)}
\sum_{\substack{A\subset I,\;|A|\leq r\\
d_A*x_{\omega_k(\mu)}\in \Wafo}}
(-1)^{|A|-\ell(d_A*\omega_k(\mu))+\ell(\omega_k(\mu))}
{g}_{u_{\overline{A}}*x_{\mu}}^{(k)}.\label{eq:12?}
\end{eqnarray}
Noting that $\omega_k(d_A*x_{\omega_k(\mu)})=
u_{\overline{A}}*x_{\mu}$, and $|\overline{A}|=|A|$, the right-hand side of \eqref{eq:12?} is identical to that of \eqref{eq:11+}
because the involution $\omega_k$ preserves the Bruhat order and the length.
{The uniqueness follows from the unitriangularity of the transition matrix and the fact that $\tilde{g}^{(k)}_{(1^r)}=\tilde{g}_{(1^r)}$ for $1\leq r\leq k$.}
\end{proof}

Theorem \ref{thm:main} will follow easily once we have the proposition below. 
\begin{prop}\label{prop:main}
For $\lambda\in \Par^k$ and $1\leq r\leq k$, we have
\begin{equation}
g_{(1^r)}\cdot \tfg{k}_\lambda
=\sum_{\mu\leqk \lambda}\sum_{\substack{A\subset I,\;|A|\leq r\\
u_A*x_\mu\in \Wafo}}
(-1)^{|A|-\ell(u_A*\mu)+\ell(\mu)}
\gcirc{k}_{u_A*x_\mu}.\label{eq:vPieri}
\end{equation}
\end{prop}

We apply the ring automorphism $\sigma$ 
of $\Lambda_{(k)}$ on both hand sides of \eqref{eq:vPieri} to obtain
\begin{equation}
\tilde{g}_{(1^r)}\cdot \sigma(\tfg{k}_\lambda)
=\sum_{\mu\leqk \lambda}\sum_{\substack{A\subset I,\;|A|\leq r\\
u_A*x_\mu\in \Wafo}}
(-1)^{|A|-\ell(u_A*\mu)+\ell(\mu)}
\sigma(\gcirc{k}_{u_A*x_\mu}).\label{eq:vP}
\end{equation}
Note that when we express $\sigma(\gcirc{k}_{u_A*x_\mu})$
as a linear combination of $\{\sigma(\tfg{k}_{\lambda})\}$,
$g_{u_A*x_\mu}^{(k)}$ is also 
a linear combination of 
$\{\tilde{g}_\lambda^{(k)}\}$ with the same coefficients.
Therefore we see that \eqref{eq:11+} and \eqref{eq:vP} are
exactly the same equation via the correspondence
$\sigma(\tfg{k}_{\lambda})\mapsto \tilde{g}^{(k)}_\lambda,$
and hence we have $  \sigma(\tfg{k}_{\lambda})=\tilde{g}^{(k)}_\lambda$
for all $\lambda\in \Par^k.$

\bigskip

Here is the outline of the proof Proposition \ref{prop:main}.
We compute ${g}_{(1^r)}\cdot \tfg{k}_\lambda$
by using the combinatorial theory of Katalan functions 
developed in \cite{BMS} to show the following
result, which is the technical heart of this paper (see \S\ref{sec:PfKey} below for the proof).  
\begin{lem}[Key lemma for Proposition \ref{prop:main}]\label{lem:main}
\begin{equation}
g_{(1^r)}\cdot \tfg{k}_\lambda=\sum_{A\subset I,\;|A|\leq r}T_{u_A}\cdot \tfg{k}_{\lambda}\label{eq:key}
\end{equation}
\end{lem}
By the definition of $\gcirc{k}_\lambda$ the right-hand side of \eqref{eq:key}
can be written as : 
\begin{equation}
\sum_{\mu\leqk \lambda}
\sum_{A\subset I,\;|A|\leq r}T_{u_A}\cdot \gcirc{k}_{\mu}.
\label{eq:key0}
\end{equation}
Thus, in order to complete the proof of Proposition \ref{prop:main}, 
it suffices to show that 
\begin{equation*}\sum_{A\subset I,\;|A|\leq r}T_{u_A}\cdot \gcirc{k}_{\mu}=
\sum_{\substack{A\subset I,\;|A|\leq r\\
u_A*x_\mu\in \Wafo}}
(-1)^{|A|-\ell(u_A*\mu)+\ell(\mu)}
\gcirc{k}_{u_A*x_\mu}.
\end{equation*}
The final step of the proof of Proposition \ref{prop:main} is the following.

\begin{lem}\label{lem:Naito2++}
Let $A\subsetneq I$. Then
\begin{equation*}
T_{u_A}\cdot\gcirc{k}_\mu
=\begin{cases}
(-1)^{|A|-\ell(u_A*x_\mu)+\ell(\mu)}
\gcirc{k}_{u_A*x_\mu}& 
(u_A*x_\mu\in \Wafo),\\
0 & (\mbox{otherwise}).
\end{cases}
\end{equation*}
\end{lem}
The proof of Lemma \ref{lem:Naito2++} is given by a general
statement (Proposition \ref{prop:generalNaito2}) that holds in the context of 
an arbitrary 
Coxeter group $(W,S)$ and its parabolic quotient.    
A detailed discussion is given in \S\ref{sec:Naito2}.

\subsection{Proof of Lemma \ref{lem:main}}\label{sec:PfKey}

This section is devoted to the proof of 
Lemma \ref{lem:main}.
\subsubsection{Basic straightening rule}
\begin{defn}[map $\mathfrak{r}$]
Define a map $\mathfrak{r}: \{1,\ldots,\ell\}\rightarrow I=\Z/(k+1)\Z$
by $\mathfrak{r}(p)=-p+1 \;\mathrm{mod}\; k+1$. 
\end{defn}
The following is the crucial combinatorial result, whose 
proof will be given in \S\ref{sec:mirror_Pf}.
\begin{lem}[Basic straightening rule]\label{lem:mirror_path} 
For $\lambda\in \Par^k_m$ and an integer $\ell-k< p\leq \ell$ with $p\geq m+2$,
\begin{eqnarray*}
\K{\Delta^k(\lambda)}{\Delta^k(\lambda)}{\lambda+\ep_p}
&=&
D_{\mathfrak{r}(p)}\cdot \tfg{k}_\lambda.
\end{eqnarray*}
Furthermore, we have 
\begin{align}
D_{\mathfrak{r}(p)}\cdot \tfg{k}_\lambda=
\tfg{k}_{s_{\mathfrak{r}(p)}x_\lambda}
&\iff
s_{\mathfrak{r}(p)}x_\lambda>x_\lambda\;\mbox{and}
\; s_{\mathfrak{r}(p)}x_\lambda
\in \Wafo\nonumber\\
&\iff
\top_{\Psi}(p)<\top_{\Psi}(p-1),\label{eq:case}
\end{align}
with $\Psi=\Delta^k(\lambda).$
\end{lem}

\begin{cor}\label{cor:mirror_path} For $\lambda\in \Par^k_m$ and an integer $\ell-k< p\leq \ell$ with $p\geq m+2$,
\begin{eqnarray*}
\K{\Delta^k(\lambda)}{L(\Delta^k(\lambda))\sqcup\{p\}}{\lambda+\ep_p}
&=&
T_{\mathfrak{r}(p)}\cdot\tfg{k}_\lambda.
\end{eqnarray*}
\end{cor}
\begin{proof}
We apply Lemma \ref{lem:dot} (i) to deduce that
\begin{eqnarray*}
&&\K{\Delta^k(\lambda)}{L(\Delta^k(\lambda))\sqcup\{p\}}{\lambda+\ep_p}\nonumber\\
&=&
\K{\Delta^k(\lambda)}{\Delta^k(\lambda)}{\lambda+\ep_p}
-
\K{\Delta^k(\lambda)}{\Delta^k(\lambda)}{\lambda}\nonumber\\
&=&D_{\mathfrak{r}(i)}\cdot\tfg{k}_{\lambda}
-\tfg{k}_\lambda\nonumber\quad(\mbox{by Lemma}\;\ref{lem:mirror_path} )\\
&=&
T_{\mathfrak{r}(i)}\cdot \tfg{k}_\lambda.
\end{eqnarray*}
\end{proof}

\ytableausetup{smalltableaux}
%\subsubsection{Examples on mirror straightening rule}

For the reader's convenience, we give some examples showing 
how the proof of Lemma \ref{lem:mirror_path} goes. 
 
\begin{example}\label{ex:strai} Let $k=4$, 
$\lambda=(3,2,1,0,0,0),$ 
$\Psi:=\Delta^k(\lambda), M=L(\Psi).$
Note that $\beta=(2,4)$ is an addable root of $\Psi$.  
Consider $\gamma=\lambda+\ep_5=(3,2,1,0,1,0).$
By using Lemma \ref{lem:addroot} (i), we have 
\begin{equation}
\K{\Psi}{M}{\gamma}=
\K{\Psi\cup \beta}{M}{\gamma}
-\K{\Psi\cup \beta}{M}{\gamma+\ep_2-\ep_4}.\label{eq:sum}
\end{equation}
$$
\begin{ytableau}
3 & & *(gray) \bullet& *(gray) \bullet & *(gray) \bullet & *(gray) \bullet\\
 & *(magenta)2 & && *(gray) \bullet & *(gray) \bullet \\
 &  & 1&&  & \\
 &  & &*(yellow){0}&  & \\
 &  && &*(magenta){1}& \\
 &  & &&  & 0 \\
\end{ytableau}
=\begin{ytableau}
3 & & *(gray) \bullet& *(gray) \bullet & *(gray) \bullet & *(gray) \bullet\\
 &{2} & &*(gray) & *(gray) \bullet & *(gray) \bullet \\
 &  & 1&&  & \\
 &  & &{0}&  & \\
 &  & && {1}& \\
 &  & &&  & 0 \\
\end{ytableau}
-\begin{ytableau}
3 & & *(gray) \bullet& *(gray) \bullet & *(gray) \bullet & *(gray) \bullet\\
 &{3} & &*(gray) & *(gray) \bullet & *(gray) \bullet \\
 &  & 1&&  & \\
 &  & &{-\!1}&  & \\
 &  & && 1& \\
 &  & &&  & 0 \\
\end{ytableau}.
$$
The first term of \eqref{eq:sum} vanishes by \eqref{eq:vanish} with $i=4$.
The second term equals 
$$(-1)^2\K{\Psi\cup \beta}{M}
{s_4(\gamma+\ep_2-\ep_4)-\ep_4+\ep_5}
=\K{\Psi\cup \beta}{M}
{\mu}
$$
with $\mu:=(3,3,1,0,0,0)$, by \eqref{eq:alt} with $i=4$.
Hence we have
$$\K{\Psi}{M}{\lambda+\ep_5}=
\K{\Psi\cup \beta}{M}{\mu}=
\begin{ytableau}
3 & & *(gray) \bullet& *(gray) \bullet & *(gray) \bullet & *(gray) \bullet\\
 &{3} & &*(gray) & *(gray) \bullet & *(gray) \bullet \\
 &  & 1&&  & \\
 &  & &0&  & \\
 &  & && 0& \\
 &  & &&  & 0 \\
\end{ytableau}.
$$
Next, we apply Lemma \ref{lem:dot} (ii) to it with $j=4$ to get
\begin{equation}
\K{\Psi\cup \beta}{M\sqcup\{4\}}
{\mu}
+
\K{\Psi\cup \beta}{M}
{(3,3,1,-\!1,0,0)}.\label{eq:sum2}
\end{equation}
$$
\begin{ytableau}
3 & & *(gray) \bullet& *(gray) \bullet & *(gray) \bullet & *(gray) \bullet\\
 &{3} & &*(gray) \bullet& *(gray) \bullet & *(gray) \bullet \\
 &  & 1&&  & \\
 &  & &0&  & \\
 &  & && 0& \\
 &  & &&  & 0 \\
\end{ytableau}
+
\begin{ytableau}
3 & & *(gray) \bullet& *(gray) \bullet & *(gray) \bullet & *(gray) \bullet\\
 &{3} & &*(gray) & *(gray) \bullet & *(gray) \bullet \\
 &  & 1&&  & \\
 &  & &-\!1&  & \\
 &  & && 0& \\
 &  & &&  & 0 \\
\end{ytableau}.
$$
The second term of \eqref{eq:sum2} vanishes by \eqref{eq:vanish}
with $i=4$.
Noting $\Psi\cup \beta=\Delta^4(\mu)$ and $M\sqcup\{4\}=L(\Delta^4(\mu))$, we finally obtain
$\K{\Psi\cup \beta}{M\sqcup\{4\}}{\mu}=
\K{\Delta^4(\mu)}{L(\Delta^4(\mu))}{\mu}=\tfg{4}_{(3,3,1)}.
$
$$
\begin{ytableau}
3 & & *(gray) \bullet& *(gray) \bullet & *(gray) \bullet & *(gray) \bullet\\
 &{3} & &*(gray)\bullet & *(gray) \bullet & *(gray) \bullet \\
 &  & 1&&  & \\
 &  & &0&  & \\
 &  & && 0& \\
 &  & &&  & 0 \\
\end{ytableau}
$$
\end{example}

\begin{example} Let $k=4$, 
and $(\Psi,M,\lambda)$ the same as Example \ref{ex:strai}. 
Consider $\gamma=\lambda+\ep_6=(3,2,1,0,0,1).$
We can apply Lemma \ref{lem:dot} (i) with $j=5$ to get
\begin{equation}
\K{\Psi}{M}{\gamma}=
\K{\Psi}{M\setminus\{5\}}{\gamma}
-\K{\Psi}{M\setminus\{5\}}{\gamma-\ep_5},
\label{eq:sum3}
\end{equation}
depicted by 
$$
\begin{ytableau}
3 & & *(gray) \bullet& *(gray) \bullet & *(gray) \bullet & *(gray) \bullet\\
 & *(yellow)2 & && *(gray) \bullet & *(gray) \bullet \\
 &  & 1&&  & \\
 &  & &{0}&  & \\
 &  & && *(yellow){0}& \\
 &  & &&  & *(magenta)1 \\
\end{ytableau}
=\begin{ytableau}
3 & & *(gray) \bullet& *(gray) \bullet & *(gray) \bullet & *(gray) \bullet\\
 &{2} & & & *(gray)  & *(gray) \bullet \\
 &  & 1&&  & \\
 &  & &{0}&  & \\
 &  & && {0}& \\
 &  & &&  & 1 \\
\end{ytableau}
-\begin{ytableau}
3 & & *(gray) \bullet& *(gray) \bullet & *(gray) \bullet & *(gray) \bullet\\
 &{2} & & & *(gray) & *(gray) \bullet \\
 &  & 1&&  & \\
 &  & &0&  & \\
 &  & && -\!1& \\
 &  & &&  & 1 \\
\end{ytableau}.$$
The first term vanishes by \eqref{eq:vanish} with $i=5$.
Hence we have 
$$\K{\Psi}{M}{\gamma}=-\K{\Psi}{M\setminus\{5\}}{\gamma-\ep_5}
=-\begin{ytableau}
3 & & *(gray) \bullet& *(gray) \bullet & *(gray) \bullet & *(gray) \bullet\\
 &{2} & & & *(gray)  & *(gray) \bullet \\
 &  & 1&&  & \\
 &  & &0&  & \\
 &  & && -\!1& \\
 &  & &&  & 1 \\
\end{ytableau}
=\begin{ytableau}
3 & & *(gray) \bullet& *(gray) \bullet & *(gray) \bullet & *(gray) \bullet\\
 &{2} & & & *(gray)  & *(gray) \bullet \\
 &  & 1&&  & \\
 &  & &0&  & \\
 &  & && 0& \\
 &  & &&  & 0 \\
\end{ytableau},
$$
where we used \eqref{eq:alt} for the last equality with $i=5$.
Thus, we obtain
$$
\begin{ytableau}
3 & & *(gray) \bullet& *(gray) \bullet & *(gray) \bullet & *(gray) \bullet\\
 &{2} & & & *(gray)  & *(gray) \bullet \\
 &  & 1&&  & \\
 &  & &0&  & \\
 &  & && 0& \\
 &  & &&  & 0 \\
\end{ytableau}
=
\begin{ytableau}
3 & & *(gray) \bullet& *(gray) \bullet & *(gray) \bullet & *(gray) \bullet\\
 &{2} & & & *(gray)\bullet  & *(gray) \bullet \\
 &  & 1&&  & \\
 &  & &0&  & \\
 &  & && 0& \\
 &  & &&  & 0 \\
\end{ytableau}
-\begin{ytableau}
3 & & *(gray) \bullet& *(gray) \bullet & *(gray) \bullet & *(gray) \bullet\\
 &{2} & & & *(gray)  & *(gray) \bullet \\
 &  & 1&&  & \\
 &  & &0&  & \\
 &  & && -\!1& \\
 &  & &&  & 0 \\
\end{ytableau}=
\begin{ytableau}
3 & & *(gray) \bullet& *(gray) \bullet & *(gray) \bullet & *(gray) \bullet\\
 &{2} & & & *(gray)\bullet  & *(gray) \bullet \\
 &  & 1&&  & \\
 &  & &0&  & \\
 &  & && 0& \\
 &  & &&  & 0 \\
\end{ytableau}=\tfg{4}_{(3,2,1)},
$$
where we again used \eqref{eq:vanish} with $i=5$ for the last equality.
\end{example}

\subsubsection{Some lemmas}

\begin{lem}[Bounce-up Lemma]\label{lem:basic}
Let $(\Psi,M,\gamma)$ be a Katalan triple, and  
$p\rightarrow q$ a bounce edge of $\Psi$, such that 
\begin{itemize}
\item[(a)] $\beta:=(q,p-1)$ is an addable root of $\Psi$,
\item[(b)] $\gamma_{p}=\gamma_{p-1}+1$,
\item[(c)] $m_M(p)=m_M(p-1)+1$,
\item[(d)] $\Psi$ has a wall in rows $p-1,p$. 
\end{itemize}Then, we have
\begin{eqnarray}
\K{\Psi}{M}{\gamma}
&=&\K{\Psi\cup \beta}{M}{\gamma+\ep_q-\ep_{p}}\label{eq:Bounce-up1}\\
&=&\K{\Psi\cup \beta}{M\sqcup\{p-1\}}{\gamma+\ep_q-\ep_{p}}.\label{eq:Bounce-up2}
\end{eqnarray}
\end{lem}
\begin{proof}
Since $\beta$ is a removable root, we have, by Lemma \ref{lem:addroot} (i), 
\begin{equation}\label{eq:bounce_lemma_1}
\K{\Psi}{M}{\gamma}
=\K{\Psi\cup \beta}{M}{\gamma}
-\K{\Psi\cup \beta}{M}{\gamma+\beta}.
\end{equation}
Note that $\Psi\cup\beta$ has a ceiling in columns $p-1,p$.
Hence, by (c) and (d), we can apply \eqref{eq:vanish} to the first term of \eqref{eq:bounce_lemma_1}, which then vanishes.
Applying \eqref{eq:alt} to the second term, we obtain \eqref{eq:Bounce-up1}.

By Lemma \ref{lem:dot} (ii) with $j=p-1$, the right-hand side of \eqref{eq:Bounce-up1}
% $\K{\Psi\cup \beta}{M}{\gamma+\ep_q-\ep_p}$ 
equals
$$
\K{\Psi\cup \beta}{M\sqcup\{p-1\}}{\gamma+\ep_q-\ep_p}
+
\K{\Psi\cup \beta}{M}{\gamma+\ep_q-\ep_{p-1}-\ep_p},
$$
in which the second term vanishes by \eqref{eq:vanish} with $i=p-1$.
Therefore, we obtain \eqref{eq:Bounce-up2}.
\end{proof}

\begin{example} 
The following equation is given by applying Lemma \ref{lem:basic} with $p=6$. 
Here, $q=3$ and $\beta=(3,5).$
$$
       \begin{ytableau}
3 & & *(lightgray) \bullet& *(lightgray) \bullet & *(lightgray) \bullet & *(lightgray) \bullet& *(lightgray) \bullet& *(lightgray) \bullet & *(lightgray) \bullet\\
 & *(yellow)2 & & & *(lightgray) \bullet  & *(lightgray) \bullet& *(lightgray) \bullet& *(lightgray) \bullet& *(lightgray) \bullet\\
 &  & *(magenta)2&& \beta &*(lightgray) \bullet&*(lightgray) \bullet&*(lightgray)\bullet&*(lightgray)\bullet
 \\
 &  &   &1&  &&&*(lightgray)\bullet& *(lightgray) \bullet
\\
   &  & & & *(yellow)1 &&&& *(lightgray) \bullet\\
    &  & &&    &*(magenta)2&&&  *(lightgray) \\
     &  & &&   & &0&&  \\
      &  & &&  &&&0& \\
       &  & &&  &&&&0 
       \end{ytableau}
       =       \begin{ytableau}
3 & & *(lightgray) \bullet& *(lightgray) \bullet & *(lightgray) \bullet & *(lightgray) \bullet& *(lightgray) \bullet& *(lightgray) \bullet & *(lightgray) \bullet\\
 & *(yellow)2 & & & *(lightgray) \bullet  & *(lightgray) \bullet& *(lightgray) \bullet& *(lightgray) \bullet& *(lightgray) \bullet\\
 &  & *(magenta)3&& *(lightgray) &*(lightgray) \bullet&*(lightgray) \bullet&*(lightgray)\bullet&*(lightgray)\bullet
 \\
 &  &   &1&  &&&*(lightgray)\bullet& *(lightgray) \bullet
\\
   &  & & & *(yellow)1 &&&& *(lightgray) \bullet\\
    &  & &&    &*(magenta)1&&&  *(lightgray) \\
     &  & &&   & &0&&  \\
      &  & &&  &&&0& \\
       &  & &&  &&&&0 
       \end{ytableau} =       \begin{ytableau}
3 & & *(lightgray) \bullet& *(lightgray) \bullet & *(lightgray) \bullet & *(lightgray) \bullet& *(lightgray) \bullet& *(lightgray) \bullet & *(lightgray) \bullet\\
 & *(yellow)2 & & & *(lightgray) \bullet  & *(lightgray) \bullet& *(lightgray) \bullet& *(lightgray) \bullet& *(lightgray) \bullet\\
 &  & *(magenta)3&& *(lightgray)\bullet &*(lightgray) \bullet&*(lightgray) \bullet&*(lightgray)\bullet&*(lightgray)\bullet
 \\
 &  &   &1&  &&&*(lightgray)\bullet& *(lightgray) \bullet
\\
   &  & & & *(yellow)1 &&&& *(lightgray) \bullet\\
    &  & &&    &*(magenta)1&&&  *(lightgray) \\
     &  & &&   & &0&&  \\
      &  & &&  &&&0& \\
       &  & &&  &&&&0 
       \end{ytableau}.
       $$
\end{example}

\begin{lem}[Absorption Lemma, {\cite[Lemma 4.4]{BMS}}]\label{lem:case2} Let $(\Psi,M,\gamma)$ be a Katalan triple, and 
$p\geq 2$ such that 
\begin{itemize} 
\item[(a)] $\top_{\Psi}(p)=p$,
\item[(b)] $\gamma_{p}=\gamma_{p-1}+1$,
\item[(c)] $m_M(p)=m_M(p-1)$,
 \item[(d)] $\Psi$ has a wall in rows $p-1,p$.
\end{itemize}
 Then, 
\begin{equation*}
\K{\Psi}{M}{\gamma}
=\K{\Psi}{M}{\gamma-\ep_{p}}.
\end{equation*}
\end{lem}

\begin{proof}
By (a), $\Psi$ has a ceiling in columns $p-1,p$.
From Lemma \ref{lem:dot} (i) with $j=p-1$, it follows that
\begin{equation}
\K{\Psi}{M}{\gamma}
=\K{\Psi}{M\setminus \{p-1\}}{\gamma}
-\K{\Psi}{M\setminus \{p-1\}}{\gamma-\ep_{p-1}}.
\label{eq:case2}
\end{equation}
Let $M'=M\setminus\{p-1\}$. 
Then, we have $m_{M'}(p)=m_{M'}(p-1)+1$, which implies that we can apply \eqref{eq:vanish} with $i=p-1$ to the first term of \eqref{eq:case2}.
Hence we obtain 
$$\K{\Psi}{M}{\gamma}=-\K{\Psi}{M'}{\gamma-\ep_{p-1}}
=\K{\Psi}{M'}{\gamma-\ep_p},$$
where we used \eqref{eq:alt} with $i=p-1$ for the last equality.
By Lemma \ref{lem:dot} (ii) with $j=p-1$, we deduce that 
$$
\K{\Psi}{M'}{\gamma-\ep_p}
=\K{\Psi}{M}{\gamma-\ep_p}
+\K{\Psi}{M'}{\gamma-\ep_p-\ep_{p-1}},
$$
in which the second term vanishes by \eqref{eq:vanish} with $i=p-1$.
\end{proof}

\begin{example}From Lemma \ref{lem:case2} with $p=6$, we have
$$
\begin{ytableau}
3 & & *(gray) \bullet& *(gray) \bullet & *(gray) \bullet & *(gray) \bullet\\
 & *(yellow)2 & && *(gray) \bullet & *(gray) \bullet \\
 &  & 1&&  & \\
 &  & &{0}&  & \\
 &  & && *(yellow){0}& \\
 &  & &&  & *(magenta)1 \\
\end{ytableau}=\begin{ytableau}
3 & & *(gray) \bullet& *(gray) \bullet & *(gray) \bullet & *(gray) \bullet\\
 & *(yellow)2 & && *(gray) \bullet & *(gray) \bullet \\
 &  & 1&&  & \\
 &  & &{0}&  & \\
 &  & && *(yellow){0}& \\
 &  & &&  & *(magenta)0 \\
\end{ytableau}.
$$
\end{example}
\subsubsection{Mirror edges and mirror paths}

\begin{defn}[Mirror edge, mirror path, mirror top] Let $\Psi$ be a root ideal, and
$e: =(p\rightarrow q)$ a bounce edge of $\Psi$.
We say that $e$ is a {\it mirror edge\/} if
$p-1\rightarrow q-1$ is also a bounce edge of $\Psi$. 
A bounce path $p=p_0\rightarrow p_1\rightarrow\cdots \rightarrow p_L$ is a {\it mirror path\/} of {\it length\/} $L$ of $\Psi$
if $p_{i}\rightarrow p_{i+1}$ is a mirror edge for each $0\leq i\leq L-1$.
If such $L$ is maximal, then we define $\mtop_{\Psi}(p)=p_{L}$ and call it the {\it mirror top\/} of $p.$ 
In particular, if $p$ is not contained in any mirror path, we have $\mtop_{\Psi}(p)=p$.
\end{defn}

\begin{example}
For the root ideal $\Psi$ illustrated in the following picture, $10\rightarrow 6\rightarrow 3$ is
       a mirror path having the maximal length $2$.
       Hence, we have $\mtop_{\Psi}(10)=3.$
$$
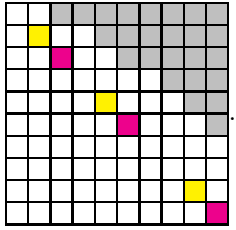
\begin{ytableau}
 {}& & *(lightgray) & *(lightgray)& *(lightgray) & *(lightgray) & *(lightgray) & *(lightgray) & *(lightgray) & *(lightgray) \\
 & *(yellow) & & & *(lightgray) & *(lightgray) & *(lightgray) & *(lightgray) & *(lightgray) & *(lightgray) \\
 &  & *(magenta)&&  &*(lightgray) &*(lightgray) &*(lightgray)&*(lightgray)
 & *(lightgray)  \\
 &  &   &&  &&&*(lightgray)& *(lightgray) 
 & *(lightgray) \\
   &  & & & *(yellow) &&&& *(lightgray) & *(lightgray) \\
    &  & &&    &*(magenta)&&&  & *(lightgray)\\
     &  & &&   & &&&&  \\
      &  & &&  &&&& &\\
       &  & &&  &&&&*(yellow)& \\ 
       &  & &&  &&&&&*(magenta)   \\ 
       \end{ytableau}.$$
\end{example}

It is easy to see  the following. 
\begin{lem}\label{lem:easy}
Let $\Psi$ be a root ideal, and $p\rightarrow q$ a mirror edge
of $\Psi$. Then $\beta:=(q,p-1)$ is an addable root of $\Psi.$
Furthermore, $\Psi\cup \{\beta\}$ 
has a wall in rows $q-1,q$, and a ceiling in columns $p-1,p$. 
\end{lem}

\begin{lem}\label{lem:max} Let $\Psi$ be a root ideal, and $z$ the lowest nonempty row of $\Psi$.
Assume that $\Psi$ is wall-free in nonempty rows, \textit{i.e.,}
$\Psi$ has no wall between the $1$st and $z$-th rows.
Then, for $p\geq 2$, we have
\begin{equation*}
\mtop_\Psi(p)=\max\{\top_\Psi(p),\top_\Psi(p-1)+1
\}.
\end{equation*}
\end{lem}
\begin{proof}
Let $p\rightarrow q$ be a bounce edge of $\Psi$.
Note that, whenever $\up_\Psi(p-1)$ exists, it satisfies $\up_\Psi(p-1)=q-1$; otherwise, $\Psi$ has a wall in rows $q-1,q$.

Let $p=p_0\rightarrow p_1\rightarrow \cdots\rightarrow
p_L$ and $p-1=q_0\rightarrow q_1\rightarrow \cdots\rightarrow
q_{L'}$ be bounce paths of maximal lengths. 
By the above remark, we have 
$q_i=p_i-1$ for $0\leq i\leq \min(L,L').$
If $L>L'$, than we have $\mtop_\Psi(p)=p_{L'}=q_{L'}+1
=\top_\Psi(p-1)+1$, while $\mtop_\Psi(p)>\top_\Psi(p)$. 
If $L=L'$, then $\mtop_\Psi(p)=\top_\Psi(p)=\top_\Psi(p-1)+1.$
If $L<L'$, then $\mtop_\Psi(p)=\top_\Psi(p),$ while 
$\mtop_\Psi(p)=\top_\Psi(p)=q_{L}+1>\top_\Psi(p-1)+1.$
\end{proof}

Note that $\Psi=\Delta^k(\lambda)$ is wall-free in nonempty rows, then we can apply the above result to it. 

\begin{lem}\label{lem:special}
Let $\lambda\in \Par^k$. 
Set $\Psi:=\Delta^k(\lambda)$, $M=L(\Psi)$.
Let $e=(p\rightarrow q)$ be a bounce edge of  $\Psi$ such that $q\geq 2$.

(1) 
$e$ is a mirror edge of $\Psi$
if and only if $\lambda_{q-1}=\lambda_{q}.$

(2) If $e$ is a mirror edge of $\Psi$, then $m_M(p)=m_M(p-1)+1$.

(3) If $e$ is not a mirror edge, then  
$\lambda_{q-1}>\lambda_{q}.$
\end{lem}
\begin{proof} (1), (2) are clear from the definitions of 
$\Psi,M$. (3) If $e$ is not a mirror edge, then it follows from (1) that 
$\lambda_{q-1}\neq \lambda_q$. Since $\lambda$ is a partition, 
we have $ \lambda_{q-1}>\lambda_{q}.$
\end{proof}

\begin{lem}\label{lem:residue_dic}
Let $\lambda\in \Par^k$, and $\Psi:=\Delta^k(\lambda)\subset 
\Delta_\ell^+$. 
Then 
for $p\in\{1,\ldots,\ell\}$, 
\begin{equation*}
\mathfrak{r}\left(\mtop_{\Psi}(p)\right)=
\mathfrak{r}(p).
\end{equation*} 
\end{lem}
\begin{proof} This is \cite[Lemma 5.7]{BMS}.
\end{proof}

In the final step of the proof of Lemma \ref{lem:mirror_path}, we use the following.
\begin{lem}[Cleaning Lemma; {\cite[Lemma 4.7]{BMS}}]\label{lem:clean}
Let $(\Psi,M,\gamma)$ be a Katalan triple
such that
\begin{itemize}
\item[(a)] $\beta=(q,p-1)$ is a removable root of $\Psi$,
\item[(b)] $\Psi$ has a wall in rows $p-1,p$,
\item[(c)] $\Psi$ has a ceiling in columns $p-1,p$,
\item[(d)] $\gamma_{q}=\gamma_{q-1}$,
\item[(e)] $m_M(p)=m_{M}(p-1)+1$.
\end{itemize}
Then 
\begin{equation}
\K{\Psi}{M}{\gamma}
=\K{\Psi\setminus \beta}{M}{\gamma}.
\end{equation}
\end{lem}

\subsubsection{Proof of Lemma \ref{lem:mirror_path}}\label{sec:mirror_Pf}

\begin{proof}[Proof of Lemma \ref{lem:mirror_path} ]
Let $\lambda\in \Par^k$, and $\Psi:=\Delta^k(\lambda)
\subset \Delta_\ell^+$. 
Set $\mathcal{T}:=(\Psi,M,\gamma)=
(\Delta^{k}(\lambda),L(\Delta^{k}(\lambda)),\lambda+\ep_p)$.
Let $p=p_0\rightarrow p_1\rightarrow \cdots\rightarrow p_{L}$ be a mirror path starting from $p$ of maximal length $L\geq 0$.
We will prove that we can successively apply Lemma \ref{lem:basic} to $\T$ and obtain the sequence of Katalan triples:
\begin{equation}\label{eq:sequence_of_triples}
\T=:\T_0\overset{p_0}{\longrightarrow}
\mathcal{T}_1
\overset{p_1}{\longrightarrow}
\cdots
\overset{p_{L-1}}{\longrightarrow}
\mathcal{T}_{L}=:\T'.
\end{equation}
Here $\T_i
\overset{p_{i}}{\longrightarrow}\T_{i+1}$ means that 
$\T_{i+1}$ is obtained from $\T_i$ by applying equation \eqref{eq:Bounce-up1} in Lemma \ref{lem:basic} with 
respect to the mirror edge $p_i\rightarrow p_{i+1}$ of $\T_i.$
If $L=0$, it suffices to put $\T'=\T$.
Assume $L\geq 1$.
We can check immediately that $\T=\T_0$ satisfies the assumptions (a)-(d) of Lemma \ref{lem:basic}: (a) holds by Lemma \ref{lem:easy};
(b) holds from $\lambda_{p-1}=\lambda_p=0$ (since $p\geq m+2$) and $\gamma=\lambda+\ep_p$;
(c) holds by Lemma \ref{lem:special} (2); (d) holds because there are no roots of $\Psi=\Delta^k(\lambda)$ 
in rows $p-1,p$ (since $p>\ell-k$). 
Then, we obtain a new Katalan triple $\T_1=(\Psi^{(1)},M^{(1)},\gamma^{(1)})$ by applying \eqref{eq:Bounce-up1} to $\T_0$.
When $L>1$, we have $\lambda_{p_1-1}=\lambda_{p_1}$ from Lemma \ref{lem:special} (1),
%From Lemma \ref{lem:special} (1), it follows that $\lambda_{p_1-1}=\lambda_{p_1}$, 
which implies $\gamma^{(1)}_{p_1}=\gamma^{(1)}_{p_1-1}+1$.
By Lemma \ref{lem:easy}, $\Psi^{(1)}$ has a wall in rows $p_1-1,p_1$. 
Hence, $\T_1$ satisfies (b), (d) of Lemma \ref{lem:basic}.
The remaining conditions (a) and (c) hold for $\T_1$ for the same reasons as in the case of $\T_0$: the multiset $M^{(1)}$ is the same as $M$; the root ideals $\Psi^{(1)}$ and $\Psi$ coincide with each other in columns $\leq p_1$.
Therefore, we can apply equation \eqref{eq:Bounce-up1} in Lemma \ref{lem:basic} to $\T_1$ with respect to the mirror edge 
$p_1\rightarrow p_2$, and obtain $\T_2.$
The above procedure can be repeated $L$ times to obtain the sequence in \eqref{eq:sequence_of_triples}.
Let $\T'=(\Psi',M',\gamma')$. Then, we have
\begin{equation}
\Psi'=\Psi\cup \{\beta_1,\ldots,\beta_{L}\},\quad
M'=M,\quad
\gamma'=\gamma+\ep_{\mtop_\Psi(p)}-\ep_p
(=\lambda+\ep_{\mtop_\Psi(p)}),\label{eq:T'}
\end{equation}
where $\beta_i:=(p_{i},p_{i-1}-1)\;(1\leq i\leq L).$
Note that $\Psi'$ has a wall in rows $p_i-1,p_i$ for each $i$.
The following procedure depends on whether (i) $\top_\Psi(p)<\top_\Psi(p-1)$ or (ii) $\top_\Psi(p)>\top_\Psi(p-1)$.

If $\top_\Psi(p)<\top_\Psi(p-1)$, then we have  $\mtop_{\Psi}(p)=\top_\Psi(p-1)+1>\top_\Psi(p)$ by Lemma \ref{lem:max}, which implies that there is a bounce edge $\mtop_\Psi(p)\rightarrow q$ of $\Psi$ 
for some $q$.
Note that $\mtop_\Psi(p)\rightarrow q$ is also a bounce edge of $\Psi'$.
Therefore, we can apply Lemma \ref{lem:basic} to $\T'$ with respect to this bounce edge.
From \eqref{eq:Bounce-up2}, it follows that
\begin{equation*}
\K{\Psi'}{M'}{\gamma'}=
\K{\Psi'\cup \beta}{M'\sqcup\{\mtop_\Psi(p)-1\}}{\gamma'+\ep_{q}-\ep_{\mtop_\Psi(p)}},
\end{equation*}
where
\begin{equation*}
\beta:=(q,\mtop_\Psi(p)-1).
\end{equation*}
Note that
\begin{equation*}
\lambda':=\gamma'+\ep_{q}-\ep_{\mtop_\Psi(p)}
=\lambda+\ep_{q}
\end{equation*}
and $M'\sqcup\{\mtop_\Psi(p)-1\}
=L(\Delta^k(\lambda')).$
Moreover, $\lambda'$ is a partition by Lemma \ref{lem:special} (3).
To prove that $\lambda'$ is $k$-bounded, it suffices to consider the case that $\lambda_1=k$ and $q=1$. 
However, one sees that this does not occur in view of  Lemma \ref{lem:max}.
We also note that
\begin{equation*}
\Psi\cup \beta=
(\Psi'\cup\beta)\setminus\{\beta_1,\ldots,\beta_{L}\}
=\Delta^k(\lambda').
\end{equation*}
By using Lemma \ref{lem:clean} recursively, we can remove 
$\beta_{L},\ldots,\beta_{2},\beta_1$ from $\Psi'$,
and obtain
$$
\K{\Delta^k(\lambda')}{L(\Delta^k(\lambda'))}{\lambda'}
=\tfg{k}_{\lambda'}.
$$
From Lemma \ref{lem:residue_dic}, we deduce that 
$$
q=\up_\Psi(\mtop_\T(p))\equiv
\mtop_\T(p)\equiv 
\mathfrak{r}(p) \;\mathrm{mod}\;k+1,
$$
and $s_{\mathfrak{r}(p)}x_\lambda
=x_{\lambda'}.$

Next assume that $\top_{\Psi}(p)>\top_{\Psi}(p-1)$. 
We check that $\T'$ satisfies conditions (a)-(d) of Lemma \ref{lem:case2} with $p=\mtop_{\Psi}(p)$.
From Lemma \ref{lem:max}, we have $\mtop_\Psi(p)=\top(p)$, which implies (a).
(b) follows from Lemma \ref{lem:special} (1).
(c) follows from the fact that there is a ceiling in the columns $p-1,p$.
(d) follows from the definition of $\Psi'$.

From Lemma \ref{lem:case2}, we have
\begin{equation*}
\K{\Psi'}{M'}{\gamma'}=
\K{\Psi'}{M'}{\gamma'-\ep_{\mtop_\Psi(p)}}.
\end{equation*}
In view of \eqref{eq:T'},
we can rewrite 
\begin{equation*}
\K{\Psi'}{M'}{\gamma'-\ep_{\mtop_\Psi(p)}}
=\K{\Delta^k(\lambda)\cup \{\beta_1,\ldots,\beta_{L}\}}{\Delta^k(\lambda)}{\lambda}.
\end{equation*}
By Lemma \ref{lem:clean}, we can remove 
$\beta_{L},\ldots,\beta_{2},\beta_1$ successively
to obtain
$$
\K{\Delta^k(\lambda)}{L(\Delta^k(\lambda))}{\lambda}
=\tfg{k}_{\lambda}.
$$
Again by Lemma \ref{lem:residue_dic}, we deduce that  
$$
q=\up_\Psi(\mtop_\T(p))\equiv
\mtop_\T(p)\equiv 
\mathfrak{r}(p) \;\mathrm{mod}\; k+1,
$$
and $s_{\mathfrak{r}(p)}x_\lambda
=x_{\lambda}.$
\end{proof}

Here are some examples.

\begin{example} Case 1 ($p=8$):  
$\top(p)=\mtop_\Psi(p)=4<\top(p-1)=7$ with $p=8$.
In this case, we have $L=0.$
\begin{enumerate}
\item Lemma \ref{lem:basic} \eqref{eq:Bounce-up2} with $p=8.$
\end{enumerate}
\ytableausetup{smalltableaux}
$$
\begin{ytableau}
3 & & *(lightgray) \bullet& *(lightgray) \bullet & *(lightgray) \bullet & *(lightgray) \bullet& *(lightgray) \bullet& *(lightgray) \bullet & *(lightgray) \bullet\\
 & 2 & & & *(lightgray) \bullet  & *(lightgray) \bullet& *(lightgray) \bullet& *(lightgray) \bullet& *(lightgray) \bullet\\
 &  & 2&&  &*(lightgray) \bullet&*(lightgray) \bullet&*(lightgray)\bullet&*(lightgray)\bullet
  \\
 &  &   &*(magenta)1&  &&&*(lightgray)\bullet& *(lightgray) \bullet
  \\
   &  & & & 1 &&&& *(lightgray) \bullet\\
    &  & &&    &1&&&  \\
     &  & &&   & &*(yellow)0&& \\
      &  & &&  &&&*(magenta)1&\\
       &  & &&  &&&&0 \\ 
       \end{ytableau}       =\begin{ytableau}
3 & & *(lightgray) \bullet& *(lightgray) \bullet & *(lightgray) \bullet & *(lightgray) \bullet& *(lightgray) \bullet& *(lightgray) \bullet & *(lightgray) \bullet\\
 & 2 & & & *(lightgray) \bullet  & *(lightgray) \bullet& *(lightgray) \bullet& *(lightgray) \bullet& *(lightgray) \bullet\\
 &  & 2&&  &*(lightgray) \bullet&*(lightgray) \bullet&*(lightgray)\bullet&*(lightgray)\bullet
  \\
 &  &   &*(magenta)2&  &&*(lightgray)\bullet&*(lightgray)\bullet& *(lightgray) \bullet
  \\
   &  & & & 1 &&&& *(lightgray) \bullet\\
    &  & &&    &1&&&  \\
     &  & &&   & &*(yellow)0&& \\
      &  & &&  &&&*(magenta)0&\\
       &  & &&  &&&&0 \\ 
       \end{ytableau}.       $$
\end{example}

\begin{example} Case 2 ($p=11$) : $\top(p)=11>\top(p-1)=6$. $L=0$.
\begin{enumerate}
\item Lemma \ref{lem:case2} with 
$p=11$.
\end{enumerate}
\ytableausetup{smalltableaux}
$$
\begin{ytableau}
3 & & *(gray) \bullet& *(gray) \bullet & *(gray) \bullet & *(gray) \bullet& *(gray) \bullet& *(gray) \bullet & *(gray) \bullet & *(gray) \bullet& *(gray) \bullet& *(gray) \bullet & *(gray) \bullet\\
 & 2 & & & *(gray) \bullet  & *(gray) \bullet& *(gray) \bullet& *(gray) \bullet& *(gray) \bullet& *(gray) \bullet& *(gray) \bullet& *(gray) \bullet&*(gray) \bullet \\
 &  & 1&&  &&*(gray) \bullet&*(gray)\bullet&*(gray)\bullet&*(gray)\bullet&*(gray)\bullet&*(gray) \bullet& *(gray) \bullet \\
 &  &   &1&  &&&*(gray)\bullet&*(gray)\bullet&*(gray)\bullet&*(gray)\bullet&*(gray) \bullet&*(gray) \bullet  \\
   &  & & &1 &&&&*(gray)\bullet&*(gray)\bullet&*(gray) \bullet&*(gray) \bullet& *(gray) \bullet \\
    &  & &&    & *(yellow)1&&&&*(gray) \bullet&*(gray) \bullet&*(gray) \bullet&*(gray) \bullet  \\
     &  & &&   & &0&&&&&*(gray) \bullet& *(gray) \bullet \\
      &  & &&  &&&0&&&&& *(gray) \bullet \\
       &  & &&  &&&&0&&&& \\ 
       &  & &&  &&&&&*(yellow)0&&&\\
       &  & &&  &&&&&&*(magenta)1&& \\
       &  & &&  &&&&&&&0& \\
       &  & &&  &&&&&&&&  0\\
\end{ytableau}
=\begin{ytableau}
3 & & *(gray) \bullet& *(gray) \bullet & *(gray) \bullet & *(gray) \bullet& *(gray) \bullet& *(gray) \bullet & *(gray) \bullet & *(gray) \bullet& *(gray) \bullet& *(gray) \bullet & *(gray) \bullet\\
 & 2 & & & *(gray) \bullet  & *(gray) \bullet& *(gray) \bullet& *(gray) \bullet& *(gray) \bullet& *(gray) \bullet& *(gray) \bullet& *(gray) \bullet&*(gray) \bullet \\
 &  & 1&&  &&*(gray) \bullet&*(gray)\bullet&*(gray)\bullet&*(gray)\bullet&*(gray)\bullet&*(gray) \bullet& *(gray) \bullet \\
 &  &   &1&  &&&*(gray)\bullet&*(gray)\bullet&*(gray)\bullet&*(gray)\bullet&*(gray) \bullet&*(gray) \bullet  \\
   &  & & &1 &&&&*(gray)\bullet&*(gray)\bullet&*(gray) \bullet&*(gray) \bullet& *(gray) \bullet \\
    &  & &&    & *(yellow)1&&&&*(gray) \bullet&*(gray) \bullet&*(gray) \bullet&*(gray) \bullet  \\
     &  & &&   & &0&&&&&*(gray) \bullet& *(gray) \bullet \\
      &  & &&  &&&0&&&&& *(gray) \bullet \\
       &  & &&  &&&&0&&&& \\ 
       &  & &&  &&&&&*(yellow)0&&&\\
       &  & &&  &&&&&&*(magenta)0&& \\
       &  & &&  &&&&&&&0& \\
       &  & &&  &&&&&&&&  0\\
\end{ytableau}.
$$
\end{example}

\begin{example} Case 1 $(p=9)$: $\top(p)=2<\top(p-1)=4.$
\begin{enumerate}
\item Lemma \ref{lem:basic} \eqref{eq:Bounce-up1} with 
$p=9$.
\item Lemma \ref{lem:basic} \eqref{eq:Bounce-up2} with $p=5.$
\item Cleaning Lemma with $q=5$.
\end{enumerate}
\ytableausetup{smalltableaux}
$$
\begin{ytableau}
3 & & *(lightgray) \bullet& *(lightgray) \bullet & *(lightgray) \bullet & *(lightgray) \bullet& *(lightgray) \bullet& *(lightgray) \bullet & *(lightgray) \bullet\\
 &*(magenta) 2 & & & *(lightgray) \bullet  & *(lightgray) \bullet& *(lightgray) \bullet& *(lightgray) \bullet& *(lightgray) \bullet\\
 &  & 2&&  &*(lightgray) \bullet&*(lightgray) \bullet&*(lightgray)\bullet&*(lightgray)\bullet
  \\
 &  &   &*(yellow)1&  &&&*(lightgray)\bullet& *(lightgray) \bullet
  \\
   &  & & & *(magenta)1 &&&& *(lightgray) \bullet\\
    &  & &&    &1&&&  \\
     &  & &&   & &0&& \\
      &  & &&  &&&*(yellow)0&\\
       &  & &&  &&&&*(magenta)1 \\ 
       \end{ytableau}
       =\begin{ytableau}
3 & & *(lightgray) \bullet& *(lightgray) \bullet & *(lightgray) \bullet & *(lightgray) \bullet& *(lightgray) \bullet& *(lightgray) \bullet & *(lightgray) \bullet\\
 & *(magenta)2 & & & *(lightgray) \bullet  & *(lightgray) \bullet& *(lightgray) \bullet& *(lightgray) \bullet& *(lightgray) \bullet\\
 &  & 2&&  &*(lightgray) \bullet&*(lightgray) \bullet&*(lightgray)\bullet&*(lightgray)\bullet
  \\
 &  &   &*(yellow)1&  &&&*(lightgray)\bullet& *(lightgray) \bullet
  \\
   &  & & & *(magenta)2 &&&*(lightgray)& *(lightgray)\bullet\\
    &  & &&    &1&&&  \\
     &  & &&   & &0&& \\
      &  & &&  &&&*(yellow)0&\\
       &  & &&  &&&&*(magenta)0 \\ 
       \end{ytableau}
       =\begin{ytableau}
3 & & *(lightgray) \bullet& *(lightgray) \bullet & *(lightgray) \bullet & *(lightgray) \bullet& *(lightgray) \bullet& *(lightgray) \bullet & *(lightgray) \bullet\\
 & *(magenta)3 & & *(lightgray)\bullet& *(lightgray) \bullet  & *(lightgray) \bullet& *(lightgray) \bullet& *(lightgray) \bullet& *(lightgray) \bullet\\
 &  & 2&&   &*(lightgray) \bullet&*(lightgray) \bullet&*(lightgray)\bullet&*(lightgray)\bullet
  \\
 &  &   &*(yellow)1&  &&&*(lightgray)\bullet& *(lightgray) \bullet
  \\
   &  & & & *(magenta)1 &&&*(lightgray)& *(lightgray)\bullet\\
    &  & &&    &1&&&  \\
     &  & &&   & &0&& \\
      &  & &&  &&&*(yellow)0&\\
       &  & &&  &&&&*(magenta)0 \\ 
       \end{ytableau}
              =\begin{ytableau}
3 & & *(lightgray) \bullet& *(lightgray) \bullet & *(lightgray) \bullet & *(lightgray) \bullet& *(lightgray) \bullet& *(lightgray) \bullet & *(lightgray) \bullet\\
 & *(magenta)3 & & & *(lightgray) \bullet  & *(lightgray) \bullet& *(lightgray) \bullet& *(lightgray) \bullet& *(lightgray) \bullet\\
 &  & 2&&  &*(lightgray) \bullet&*(lightgray) \bullet&*(lightgray)\bullet&*(lightgray)\bullet
  \\
 &  &   &*(yellow)1&  &&&*(lightgray)\bullet& *(lightgray) \bullet
  \\
   &  & & & *(magenta)1 &&&& *(lightgray)\bullet\\
    &  & &&    &1&&&  \\
     &  & &&   & &0&& \\
      &  & &&  &&&*(yellow)0&\\
       &  & &&  &&&&*(magenta)0 \\ 
       \end{ytableau}.
       $$
\end{example}

\begin{example} Case 1 ($p=10$): 
$\top(p)=1<\top(p-1)=2$.
$\mathfrak{r}(p)=1$ 
\begin{enumerate}
\item Lemma \ref{lem:basic} \eqref{eq:Bounce-up1} with 
$p=10$.
\item Lemma \ref{lem:basic} \eqref{eq:Bounce-up1} with 
$p=6$.
\item Lemma \ref{lem:basic} \eqref{eq:Bounce-up2} with $p=3.$
\item Cleaning Lemma twice with $q=3,5.$
\end{enumerate}
\ytableausetup{smalltableaux}
$$
\begin{aligned}
\begin{ytableau}
*(magenta)3 & & *(lightgray) \bullet& *(lightgray) \bullet & *(lightgray) \bullet & *(lightgray) \bullet& *(lightgray) \bullet& *(lightgray) \bullet & *(lightgray) \bullet& *(lightgray) \bullet\\
 & *(yellow)2 & & & *(lightgray) \bullet  & *(lightgray) \bullet& *(lightgray) \bullet& *(lightgray) \bullet& *(lightgray) \bullet& *(lightgray) \bullet\\
 &  & *(magenta)2&&  &*(lightgray) \bullet&*(lightgray) \bullet&*(lightgray)\bullet&*(lightgray)\bullet
 & *(lightgray) \bullet \\
 &  &   &1&  &&&*(lightgray)\bullet& *(lightgray) \bullet
 & *(lightgray) \bullet \\
   &  & & & *(yellow)1 &&&& *(lightgray) \bullet& *(lightgray) \bullet\\
    &  & &&    &*(magenta)1&&&  & *(lightgray) \bullet\\
     &  & &&   & &0&&&  \\
      &  & &&  &&&0& &\\
       &  & &&  &&&&*(yellow)0& \\ 
       &  & &&  &&&&&*(magenta)1   \\ 
       \end{ytableau}
       =
       \begin{ytableau}
*(magenta)3 & & *(lightgray) \bullet& *(lightgray) \bullet & *(lightgray) \bullet & *(lightgray) \bullet& *(lightgray) \bullet& *(lightgray) \bullet & *(lightgray) \bullet& *(lightgray) \bullet\\
 & *(yellow)2 & & & *(lightgray) \bullet  & *(lightgray) \bullet& *(lightgray) \bullet& *(lightgray) \bullet& *(lightgray) \bullet& *(lightgray) \bullet\\
 &  & *(magenta)2&&  &*(lightgray) \bullet&*(lightgray) \bullet&*(lightgray)\bullet&*(lightgray)\bullet
 & *(lightgray) \bullet \\
 &  &   &1&  &&&*(lightgray)\bullet& *(lightgray) \bullet
 & *(lightgray) \bullet \\
   &  & & & *(yellow)1 &&&& *(lightgray) \bullet& *(lightgray) \bullet\\
    &  & &&    &*(magenta)2&&&  *(lightgray)  & *(lightgray) \bullet\\
     &  & &&   & &0&&&  \\
      &  & &&  &&&0& &\\
       &  & &&  &&&&*(yellow)0& \\ 
       &  & &&  &&&&&*(magenta)0   \\ 
       \end{ytableau}
           =
       \begin{ytableau}
*(magenta)3 & & *(lightgray) \bullet& *(lightgray) \bullet & *(lightgray) \bullet & *(lightgray) \bullet& *(lightgray) \bullet& *(lightgray) \bullet & *(lightgray) \bullet& *(lightgray) \bullet\\
 & *(yellow)2 & & & *(lightgray) \bullet  & *(lightgray) \bullet& *(lightgray) \bullet& *(lightgray) \bullet& *(lightgray) \bullet& *(lightgray) \bullet\\
 &  & *(magenta)3&& *(lightgray) &*(lightgray) \bullet&*(lightgray) \bullet&*(lightgray)\bullet&*(lightgray)\bullet
 & *(lightgray) \bullet \\
 &  &   &1&  &&&*(lightgray)\bullet& *(lightgray) \bullet
 & *(lightgray) \bullet \\
   &  & & & *(yellow)1 &&&& *(lightgray) \bullet& *(lightgray) \bullet\\
    &  & &&    &*(magenta)1&&&  *(lightgray)  & *(lightgray) \bullet\\
     &  & &&   & &0&&&  \\
      &  & &&  &&&0& &\\
       &  & &&  &&&&*(yellow)0& \\ 
       &  & &&  &&&&&*(magenta)0   \\ 
       \end{ytableau}\\
       =       \begin{ytableau}
*(magenta)4 &*(lightgray) \bullet  & *(lightgray) \bullet& *(lightgray) \bullet & *(lightgray) \bullet & *(lightgray) \bullet& *(lightgray) \bullet& *(lightgray) \bullet & *(lightgray) \bullet& *(lightgray) \bullet\\
 & *(yellow)2 & & & *(lightgray) \bullet  & *(lightgray) \bullet& *(lightgray) \bullet& *(lightgray) \bullet& *(lightgray) \bullet& *(lightgray) \bullet\\
 &  & *(magenta)2&& *(lightgray) &*(lightgray) \bullet&*(lightgray) \bullet&*(lightgray)\bullet&*(lightgray)\bullet
 & *(lightgray) \bullet \\
 &  &   &1&  &&&*(lightgray)\bullet& *(lightgray) \bullet
 & *(lightgray) \bullet \\
   &  & & & *(yellow)1 &&&& *(lightgray) \bullet& *(lightgray) \bullet\\
    &  & &&    &*(magenta)1&&& *(lightgray)  & *(lightgray) \bullet\\
     &  & &&   & &0&&&  \\
      &  & &&  &&&0& &\\
       &  & &&  &&&&*(yellow)0& \\ 
       &  & &&  &&&&&*(magenta)0   \\ 
       \end{ytableau}
        =       \begin{ytableau}
*(magenta)4 &*(lightgray) \bullet  & *(lightgray) \bullet& *(lightgray) \bullet & *(lightgray) \bullet & *(lightgray) \bullet& *(lightgray) \bullet& *(lightgray) \bullet & *(lightgray) \bullet& *(lightgray) \bullet\\
 & *(yellow)2 & & & *(lightgray) \bullet  & *(lightgray) \bullet& *(lightgray) \bullet& *(lightgray) \bullet& *(lightgray) \bullet& *(lightgray) \bullet\\
 &  & *(magenta)2&&  &*(lightgray) \bullet&*(lightgray) \bullet&*(lightgray)\bullet&*(lightgray)\bullet
 & *(lightgray) \bullet \\
 &  &   &1&  &&&*(lightgray)\bullet& *(lightgray) \bullet
 & *(lightgray) \bullet \\
   &  & & & *(yellow)1 &&&& *(lightgray) \bullet& *(lightgray) \bullet\\
    &  & &&    &*(magenta)1&&&   & *(lightgray) \bullet\\
     &  & &&   & &0&&&  \\
      &  & &&  &&&0& &\\
       &  & &&  &&&&*(yellow)0& \\ 
       &  & &&  &&&&&*(magenta)0   \\ 
       \end{ytableau}.
\end{aligned}
           $$
\end{example}

\begin{example} Case 2 ($p=10$): $\top(p)=6>\top(p-1)=2$.
\begin{enumerate}
\item Lemma \ref{lem:basic} \eqref{eq:Bounce-up1} with 
$p=10$.
\item Lemma \ref{lem:case2} with 
$p=6$.
\item Cleaning Lemma with $q=6.$
\end{enumerate}
\ytableausetup{smalltableaux}
$$
\begin{aligned}
&
\begin{ytableau}
3 & & *(gray) \bullet& *(gray) \bullet & *(gray) \bullet & *(gray) \bullet& *(gray) \bullet& *(gray) \bullet & *(gray) \bullet & *(gray) \bullet& *(gray) \bullet& *(gray) \bullet & *(gray) \bullet\\
 &*(yellow) 2 & & & *(gray) \bullet  & *(gray) \bullet& *(gray) \bullet& *(gray) \bullet& *(gray) \bullet& *(gray) \bullet& *(gray) \bullet& *(gray) \bullet&*(gray) \bullet \\
 &  & 1&&  &&*(gray) \bullet&*(gray)\bullet&*(gray)\bullet&*(gray)\bullet&*(gray)\bullet&*(gray) \bullet& *(gray) \bullet \\
 &  &   &1&  &&&*(gray)\bullet&*(gray)\bullet&*(gray)\bullet&*(gray)\bullet&*(gray) \bullet&*(gray) \bullet  \\
   &  & & &*(yellow)1 &&&&*(gray)\bullet&*(gray)\bullet&*(gray) \bullet&*(gray) \bullet& *(gray) \bullet \\
    &  & &&    & *(magenta)1&&&&*(gray) \bullet&*(gray) \bullet&*(gray) \bullet&*(gray) \bullet  \\
     &  & &&   & &0&&&&&*(gray) \bullet& *(gray) \bullet \\
      &  & &&  &&&0&&&&& *(gray) \bullet \\
       &  & &&  &&&&*(yellow)0&&&& \\ 
       &  & &&  &&&&&*(magenta)1&&&\\
       &  & &&  &&&&&&0&& \\
       &  & &&  &&&&&&&0& \\
       &  & &&  &&&&&&&&  0\\
\end{ytableau}
=
\begin{ytableau}
3 & & *(gray) \bullet& *(gray) \bullet & *(gray) \bullet & *(gray) \bullet& *(gray) \bullet& *(gray) \bullet & *(gray) \bullet & *(gray) \bullet& *(gray) \bullet& *(gray) \bullet & *(gray) \bullet\\
 &*(yellow) 2 & & & *(gray) \bullet  & *(gray) \bullet& *(gray) \bullet& *(gray) \bullet& *(gray) \bullet& *(gray) \bullet& *(gray) \bullet& *(gray) \bullet&*(gray) \bullet \\
 &  & 1&&  &&*(gray) \bullet&*(gray)\bullet&*(gray)\bullet&*(gray)\bullet&*(gray)\bullet&*(gray) \bullet& *(gray) \bullet \\
 &  &   &1&  &&&*(gray)\bullet&*(gray)\bullet&*(gray)\bullet&*(gray)\bullet&*(gray) \bullet&*(gray) \bullet  \\
   &  & & &*(yellow)1 &&&&*(gray)\bullet&*(gray)\bullet&*(gray) \bullet&*(gray) \bullet& *(gray) \bullet \\
    &  & &&    & *(magenta)2&&&*(gray)&*(gray) \bullet&*(gray) \bullet&*(gray) \bullet&*(gray) \bullet  \\
     &  & &&   & &0&&&&&*(gray) \bullet& *(gray) \bullet \\
      &  & &&  &&&0&&&&& *(gray) \bullet \\
       &  & &&  &&&&*(yellow)0&&&& \\ 
       &  & &&  &&&&&*(magenta)0&&&\\
       &  & &&  &&&&&&0&& \\
       &  & &&  &&&&&&&0& \\
       &  & &&  &&&&&&&&  0\\
\end{ytableau}\\
&=
\begin{ytableau}
3 & & *(gray) \bullet& *(gray) \bullet & *(gray) \bullet & *(gray) \bullet& *(gray) \bullet& *(gray) \bullet & *(gray) \bullet & *(gray) \bullet& *(gray) \bullet& *(gray) \bullet & *(gray) \bullet\\
 &*(yellow) 2 & & & *(gray) \bullet  & *(gray) \bullet& *(gray) \bullet& *(gray) \bullet& *(gray) \bullet& *(gray) \bullet& *(gray) \bullet& *(gray) \bullet&*(gray) \bullet \\
 &  & 1&&  &&*(gray) \bullet&*(gray)\bullet&*(gray)\bullet&*(gray)\bullet&*(gray)\bullet&*(gray) \bullet& *(gray) \bullet \\
 &  &   &1&  &&&*(gray)\bullet&*(gray)\bullet&*(gray)\bullet&*(gray)\bullet&*(gray) \bullet&*(gray) \bullet  \\
   &  & & &*(yellow)1 &&&&*(gray)\bullet&*(gray)\bullet&*(gray) \bullet&*(gray) \bullet& *(gray) \bullet \\
    &  & &&    & *(magenta)1&&&*(gray)&*(gray) \bullet&*(gray) \bullet&*(gray) \bullet&*(gray) \bullet  \\
     &  & &&   & &0&&&&&*(gray) \bullet& *(gray) \bullet \\
      &  & &&  &&&0&&&&& *(gray) \bullet \\
       &  & &&  &&&&*(yellow)0&&&& \\ 
       &  & &&  &&&&&*(magenta)0&&&\\
       &  & &&  &&&&&&0&& \\
       &  & &&  &&&&&&&0& \\
       &  & &&  &&&&&&&&  0\\
\end{ytableau}
=
\begin{ytableau}
3 & & *(gray) \bullet& *(gray) \bullet & *(gray) \bullet & *(gray) \bullet& *(gray) \bullet& *(gray) \bullet & *(gray) \bullet & *(gray) \bullet& *(gray) \bullet& *(gray) \bullet & *(gray) \bullet\\
 &*(yellow) 2 & & & *(gray) \bullet  & *(gray) \bullet& *(gray) \bullet& *(gray) \bullet& *(gray) \bullet& *(gray) \bullet& *(gray) \bullet& *(gray) \bullet&*(gray) \bullet \\
 &  & 1&&  &&*(gray) \bullet&*(gray)\bullet&*(gray)\bullet&*(gray)\bullet&*(gray)\bullet&*(gray) \bullet& *(gray) \bullet \\
 &  &   &1&  &&&*(gray)\bullet&*(gray)\bullet&*(gray)\bullet&*(gray)\bullet&*(gray) \bullet&*(gray) \bullet  \\
   &  & & &*(yellow)1 &&&&*(gray)\bullet&*(gray)\bullet&*(gray) \bullet&*(gray) \bullet& *(gray) \bullet \\
    &  & &&    & *(magenta)1&&&&*(gray) \bullet&*(gray) \bullet&*(gray) \bullet&*(gray) \bullet  \\
     &  & &&   & &0&&&&&*(gray) \bullet& *(gray) \bullet \\
      &  & &&  &&&0&&&&& *(gray) \bullet \\
       &  & &&  &&&&*(yellow)0&&&& \\ 
       &  & &&  &&&&&*(magenta)0&&&\\
       &  & &&  &&&&&&0&& \\
       &  & &&  &&&&&&&0& \\
       &  & &&  &&&&&&&&  0\\
\end{ytableau}.
\end{aligned}
$$
\end{example}

\begin{rem}\label{rem:slight_improvement}
From the proof of Lemma \ref{lem:mirror_path} given above, we can slightly improve the statement of the lemma as follows:
For $\lambda\in \Par^k_m$, an integer $\ell-k< p\leq \ell$ with $p\geq m+2$, and $\kappa\in \mathbb{Z}^\ell$ with $p<\mathrm{supp}(\kappa)\leq \ell$, we have
\begin{eqnarray*}
\K{\Delta^k(\lambda)}{\Delta^k(\lambda)}{\lambda+\ep_p+\kappa}
&=&
\K{\Delta^k(\mu)}{\Delta^k(\mu)}{\mu+\kappa},
\end{eqnarray*}
where $\tfg{k}_\mu=D_{\mathfrak{r}(p)}\cdot \tfg{k}_\lambda.$
Indeed, throughout the straightening process introduced in  \S\ref{sec:mirror_Pf}, $\kappa$ is left unchanged because the process affects only topmost $p$ rows of the diagram.
\end{rem}

From Remark \ref{rem:slight_improvement}, we obtain the following key formula:
For $\lambda\in \Par^k_m$, and a set of integers
$\ell-k< p_1<\dots<p_r\leq \ell$ with $p_1\geq m+2$, we have
\begin{eqnarray}\label{eq:key_formula}
\K{\Delta^k(\lambda)}{\Delta^k(\lambda)}{\lambda+\ep_{p_1}+\dots+\ep_{p_{r}}}
&=&
D_{\mathfrak{r}(p_r)}\cdots D_{\mathfrak{r}(p_1)}\cdot \tfg{k}_\lambda.
\end{eqnarray}

\subsubsection{Proof of Lemma \ref{lem:main}}
Let $\Psi_1,\Psi_2$ be root ideals of $\Delta_{\ell_1}^+$ 
and  $\Delta_{\ell_2}^+$, respectively.
Let $\Psi_1\uplus \Psi_2$ be the subset of $\Delta_{\ell_1+\ell_2}^+$ 
defined by: 
\begin{equation*}
\{(i,j)\in \Delta_{\ell_1+\ell_2}^+
\;|\;(i,j)\in \Psi_1\;
\mbox{or}\;(i-\ell_1,j-\ell_1)\in \Psi_2
\;\mbox{or}\;
(i\leq \ell_1 \;\mbox{and}\; j>\ell_1)
\}.
\end{equation*}

Let $\lambda\in \Par^k_\ell.$ 
Set $\Psi=\Delta^k(\lambda).$ By the product rule \cite[Lemma 3.8]{BMS}, we have
\begin{eqnarray}
g_{(1^r)}\cdot \tfg{k}_\lambda
&=&\K{\Psi\uplus \varnothing_r}{\Psi\uplus \varnothing_r}{\lambda+\ep_{\{\ell+1,\ldots,\ell+r\}}}.\label{eq:Pieri1}
\end{eqnarray}
By applying Diagonal Removable Lemma \cite[Lemma 4.13]{BMS},
one shows that the right-hand side of \eqref{eq:Pieri1} is identical to
\begin{equation}
\K{\Psi'}{\Psi'}{\lambda+\ep_{\{\ell+1,\ldots,\ell+r\}}},\label{eq:Pieri2}
\end{equation}
where
$\Psi'=(\Delta^k(\lambda)\uplus\varnothing_r)\setminus \{(i,j)\in \Delta_{\ell+r}^+\;|\;i\leq \ell,\;j\geq \ell,j-i\leq r-1
\}$.
We can apply \cite[Lemma 5.3]{BMS}
with 
{$x=\ell-k+1$, $h=k-r+1$}
(in the notation there) to show that
\eqref{eq:Pieri2} is identical to
\begin{equation}
\sum_{a=0}^r
\sum_{\substack{S\subset \{\ell-k+1+r-a,\ldots,\ell\}\\
 |S|=a}}
 \K{\Psi''}{L(\Psi'')\sqcup S}{\lambda+\ep_S+\ep_{\{\ell+1,\ldots,\ell+r-a\}}},\label{eq:K_Pieri}
\end{equation}
 where
$
\Psi''=\Psi'\setminus
\{(i,j)\;|\;
i\leq\ell,j\geq \ell,j-i\leq k\}
 = \Delta^k(\lambda)\subset \Delta^+_{\ell+r}$
(see also the last part of the proof of \cite[Proposition 5.4]{BMS}).
By 
Lemma \ref{lem:dot} (i)
%}%
and \eqref{eq:key_formula}, 
{
for $S=\{p_1<p_2<\dots<p_a\}$,
we have
\[
\begin{aligned}
&\K{\Psi''}{L(\Psi'')\sqcup S}{\lambda+\ep_S+\ep_{\{\ell+1,\ldots,\ell+r-a\}}}\\
&=
\sum_{S'\subset S}
(-1)^{|S\setminus S'|}
\K{\Psi''}{L(\Psi'')}{\lambda+\ep_{S'}+\ep_{\{\ell+1,\ldots,\ell+r-a\}}}\\
&=
\sum_{b=0}^a
\sum_{\ell-k+1+r-a\leq 
p_{i_1}<p_{i_2}<\dots<p_{i_b}\leq \ell
}
(-1)^{a-b}
D_{\mathfrak{r}(\ell+r-a)}\cdots %D_{\mathfrak{r}(\ell+2)}
D_{\mathfrak{r}(\ell+1)}
D_{\mathfrak{r}(p_{i_b})}\cdots %T_{\mathfrak{r}(p_2)}
D_{\mathfrak{r}(p_{i_1})}\cdot \tfg{k}_\lambda\\
&=
D_{\mathfrak{r}(\ell+r-a)}\cdots %D_{\mathfrak{r}(\ell+2)}
D_{\mathfrak{r}(\ell+1)}
T_{\mathfrak{r}(p_a)}\cdots %T_{\mathfrak{r}(p_2)}
T_{\mathfrak{r}(p_1)}\cdot \tfg{k}_\lambda.
\end{aligned}
\]
Therefore, \eqref{eq:K_Pieri} is equal to
}
\begin{equation}
\sum_{a=0}^r
\sum_{\ell-k+1+r-a\leq 
p_1<\cdots<p_a\leq \ell}
D_{\mathfrak{r}(\ell+r-a)}\cdots %D_{\mathfrak{r}(\ell+2)}
D_{\mathfrak{r}(\ell+1)}
T_{\mathfrak{r}(p_a)}\cdots %T_{\mathfrak{r}(p_2)}
T_{\mathfrak{r}(p_1)}\cdot \tfg{k}_\lambda.\label{eq:DT}
\end{equation}

\begin{prop}
For $1\leq r\leq k$, and $\ell\in \Z$, we have
\[
\sum_{a=0}^r\sum_{\ell-k+1+r-a\leq p_1<\dots<p_a\leq \ell}
D_{\mathfrak{r}(\ell+r-a)}\dots D_{\mathfrak{r}(\ell+1)}
T_{\mathfrak{r}(p_a)}\dots T_{\mathfrak{r}(p_1)}
=
\sum_{A\subset I,\ |A|\leq r}T_{u_A}.
\]
\end{prop}
\begin{proof}We first note that 
for any subsequence $A_1$ of $(\r(\ell+r-a),\ldots,\r(\ell+1))$ such that 
$0\leq |A_1|\leq r-a$, and $A_2$ of 
%$(\r(\ell-k+r-a),\ldots,\r(\ell))$ 
{$(\r(\ell),\r(\ell-1),\dots,\r(\ell-k+1+r-a) )$ }
such that $|A_2|=a$,  
the concatenation $A_1\cdot A_2$ is a cyclically
increasing sequence of length less than or equal to $r.$  

For $0\leq a\leq n\leq r$, let
\[
X^n_a=\left\{(p_1,\dots,p_n)\in \Z^n\ \left|\
   \begin{aligned}
   &\ell-k+r-a< p_1<\dots<p_a\leq \ell\\
   &<p_{a+1}<\dots<p_n\leq \ell+r-a
   \end{aligned}\
\right\}\right.
\]
Substituting $D_p=T_p+1$ into the left-hand side of the equation in the proposition, we see that
\begin{align}
&\sum_{a=0}^r\sum_{\ell-k+1+r-a\leq p_1<\dots<p_a\leq \ell}
(T_{\mathfrak{r}(\ell+r-a)}+1)\dots (T_{\mathfrak{r}(\ell+1)}+1)
T_{\mathfrak{r}(p_a)}\dots T_{\mathfrak{r}(p_1)}\nonumber\\
&=\sum_{a=0}^r\sum_{n=a}^r
\sum_{
(p_1,\dots,p_n)\in X^n_a
}
T_{u_{\{\mathfrak{r}(p_1),\dots,\mathfrak{r}(p_n)\}}}\nonumber\\
&=\sum_{n=0}^r\sum_{a=0}^n
\sum_{
(p_1,\dots,p_n)\in X^n_a
}
T_{u_{\{\mathfrak{r}(p_1),\dots,\mathfrak{r}(p_n)\}}}.\nonumber
\end{align}
Then it suffices to show that the map
\[
\mathrm{rm}:\bigsqcup_{a=0}^nX^n_a\to \{A\subset I\,;\, |A|=n\},\quad
(p_1,\dots,p_n)\mapsto \{\overline{p_1},\dots,\overline{p_n}\},
\]
is bijective.
For this, we identify a subset $A\subset I$ with a 01-sequence $\eta_1\eta_2\cdots \eta_{k+1}\in \{0,1\}^{k+1}$ by letting $\eta_p=0$ if $\overline{p+\ell}\in A$ and $\eta_p=1$ otherwise.
Then the image of $X^n_a$ by $\mathrm{rm}$ is contained in the set
\[
S^n_a:=
\left\{
\eta_1\eta_2\dots \eta_{k+1}\in \{0,1\}^{k+1}\,\left|\,
   \begin{aligned}
   \textstyle\sum_{p=1}^{k+1}\eta_p&=k+1-n,\\
 \textstyle\sum_{p=1}^{r-a}\eta_{p}&=r-n\\
   \eta_{r-a+1}&=1
%   &\textstyle\sum_{i=1}^{r-a}\eta_{\overline{\ell+i}}=r-n\\
%   &\textstyle\sum_{i=1}^{r-a+1}\eta_{\overline{\ell+i}}
%   =r-n+1
   \end{aligned}
\right\}\right..
\]
%Obviously, 
For any $\eta_1\eta_2\dots \eta_{k+1}\in \bigcup_{a=0}^nS_a^n$, we have
\[
a=r+1-\min\{p\,|\,\eta_1+\dots+\eta_p=r-n+1\},
\]
which implies $S^n_a\cap S^{n}_{a'}= \emptyset\iff a\neq a'$. 
This implies that $\mathrm{rm}$ is injective as its restriction to $X^n_a$ is injective.
Because $\sum_{a=0}^n|X^n_a|=\sum_{a=0}^n\binom{k-r+a}{a}\binom{r-a}{n-a}=\binom{k+1}{n}=\sharp\{A\subset I\,;\,|A|=n\}$, the map $\mathrm{rm}$ is also surjective.
\end{proof}

This completes the proof of Lemma \ref{lem:main}.

\section{Relation to quantum $K$-theory}\label{sec:QK}

 \subsection{Quantum $K$-theory ring of the flag variety}\label{sec:QK}

\begin{thm}[\cite{KM},\cite{LM},\cite{LNS},\cite{MNS1},\cite{MNS2}]\label{thm:QK}
 ${QK}(\Fl)$
 can be identified with the quotient ring
\begin{equation}
A_{k+1}:=\C[[Q]][z_1,\ldots,z_{k+1}]/I_{k+1},\label{eq:A}
\end{equation}
where  $I_{k+1}$ is the ideal of $\C[[Q]][z_1,\ldots,z_{k+1}]$ generated by 
\begin{equation}
\sum_{\substack{I\subset \{1,\ldots,k+1\}\\|I|=i}}\prod_{j\in I}z_j\!
\prod_{\substack{j\in I\\ j+1\notin I}}(1-Q_j)
-\textstyle{\binom{k+1}{i}}\quad(1\leq i\leq k+1),\label{eq:QKrel}
\end{equation}
such that
for $w\in S_{k+1}$ the {\it quantum Grothendieck polynomial\/}
$\G_w^Q$ of Lenart and Maeno, with the change of variables $x_i=1-z_i$, represents $\O_{\Fl}^w$ 
in ${QK}(\Fl_{k+1})$.
\end{thm}

Let $QK^\pol(\Fl)$ denote the
$\C[Q]$-module spanned by $\O^w_{\Fl}\;(w\in S_{k+1})$.
Let $S$ be the multiplicative subset $1+(Q)$ of $\C[Q].$ For a $\C[Q]$-module $M$,
 we denote by $M_S$ the localization by $S.$
 By a result of 
 Anderson, Chen, Tseng, and Iritani \cite{ACTI}, $QK^\pol(\Fl)_S$
  forms a subring of $QK(G/B)$.

% polynomial version of 
% the quantum $K$-theory ring 
%  of the flag variety $\Fl$.
%  This is a free $\C[Q]$-module with 
%  basis 
%  inside the (small) quantum $K$-theory ring ${QK}(\Fl)$ (\cite{GL}).
%  Anderson, Chen, Tseng, and Iritani  
%  \cite{ACTI} proved that $QK^\pol(\Fl)$
%  forms a subring. In type A, this fact can also be verified by using an explicit formula  in \cite[Prop. 3.1]{MNS2} expressing $\O^w_{G/B}$ in terms of the classes of the tautological
% line bundles.  
Let $A_{k+1}^{\mathrm{pol}}$ be the quotient ring
$
\C[Q][z_1,\ldots,z_{k+1}]/I_{k+1}^{\mathrm{pol}}$, where 
$I_{k+1}^{\mathrm{pol}}$ is the ideal of $\C[Q][z_1,\ldots,z_{k+1}]$ generated by
the polynomials \eqref{eq:QKrel}. 
The following result is a modified version of Kirillov-Maeno's (conjectural) presentation
given in \cite[Remark in \S6.1]{LNS}. For the reader's convenience, we give a sketch of proof.
\begin{prop}\label{prop:KMmodified}
There exists an isomorphism of $\C[Q]_S$
algebras \begin{equation}
% \C[Q][z_1,\ldots,z_{k+1}]_S/I_{k+1,S}^{\mathrm{pol}}
(A^\pol_{k+1})_S
\longrightarrow
{QK}^{\mathrm{pol}}(\Fl)_S.
\label{eq:QK=CQ}
\end{equation}
Moreover the image of $\G_w^Q$ in $(A^\pol_{k+1})_S$ is sent to $\O^{w}_{\Fl}$
for $w\in S_{k+1}.$
%such that
%for $w\in S_{k+1}$ the {\it quantum Grothendieck polynomial\/}
%$\G_w^Q$ of Lenart and Maeno, with the change of variables $x_i=1-z_i$, represents $\O_{\Fl}^w$ 
%in $\hat {QK}(\Fl_{k+1})$.
\end{prop}
\begin{proof} 
    By \cite[Chapter 3, Exercise 2]{AM}, we see that the localized ideal 
    $(Q)_S$ is contained in the Jacobson radical of $\C[Q]_S$. 
    By \cite[Corollary B.3(2)]{MNS1}, we see that the right-hand side of \eqref{eq:QK=CQ} is finitely generated as $\C[Q]_S$-module.
    Then we can apply Nakayama-type arguments (\cite[Proposition A.3 and Remark A.6]{GMSZ}) to obtain the isomorphism by the same arguments as in \cite[Theorem 6.1]{MNS1}. 
    The second statement follows from 
 \cite[Theorem 50]{LNS} or
 \cite[Theorem 4.4]{MNS2}.
\end{proof}

 \subsection{Map $\Phi_{k+1}$}

Recall that $\tau_i:=g_{R_i},\;\tau_i^+:=\tilde{g}_{R_i}
$ with $R_i= (\overbrace{i,\ldots,i}^{k+1-i})$ for $1\leq i\leq k$.
Note that 
the notation $R_i$ is different from the one used in \cite{IIM}, and
the indices of $\tau_i$ are also switched from the ones in \cite{IIM} by $i\mapsto k+1-i$ (see Remark \ref{rem:omega} below). 

\begin{thm}[\cite{IIM}]\label{thm:IIM}
There is a ring isomorphism
\begin{equation}
\Phi_{k+1}: A_{k+1}^{\mathrm{pol}}[Q_i^{-1}]
\longrightarrow 
\Lambda_{(k)}[\tau_i^{-1},(\tau^+_i)^{-1}],
\end{equation}
where $1\le i\le k,$ such that 
\begin{equation}\label{eq:def_of_Psi}
z_i\mapsto \frac{\tau_i\tau^+_{i-1}}{\tau^+_i\tau_{i-1}}\;(1\leq i\leq k+1),\quad
Q_i\mapsto \frac{\tau_{i-1}\tau_{i+1}}{\tau_i^2}\;(1\leq i\leq k).
\end{equation}
\end{thm}
The map $\Phi_{k+1}$ was constructed by 
solving the relativistic Toda lattice equation 
with the initial condition that the Lax matrix is unipotent. 
Although the construction of $\Phi_{k+1}$ has no apparent geometric meaning, 
it is expected that 
the map sends a Schubert structure sheaf $\O_{\Fl}^w$ to 
an element in $K_*(\Gr)_\loc$ related to a Schubert class.  

\begin{rem}\label{rem:omega} Our convention for $\Phi_{k+1}$ is slightly 
different from the one in \cite{IIM}. 
Let $\overline{\Omega}$ be the automorphism of $\Lambda_{(k)}[\tau_i^{-1},(\tau^+_i)^{-1}]$ given as the natural extension of $\Omega$ (see \S \ref{sec:kconj}).
One can check that $\Phi_{k+1}'=\overline{\Omega}\circ \Phi_{k+1}$ coincides with the map introduced in \cite{IIM} by replacing $\sigma_i$ in their notation with $\tau^+_i$.

\begin{cor}\label{cor:phi}
% Let 
% $A_{k+1}^\pol:=\C[Q][z_1,\ldots,z_{k+1}]_S/I_{k+1,S}^{\mathrm{pol}}$. Let $T$ be the multiplicative subset generated by $Q_1,\ldots,Q_k$.
There is an injective $\C[Q]$-algebra homomorpshism
$$\phi: A_{k+1}^\pol[Q_i^{-1}]\hookrightarrow QK(\Fl)[Q_i^{-1}]$$
such that the image of $\G_w^Q$ in $A_{k+1}^\pol[Q_i^{-1}]$ is sent to 
$\O^w_{\Fl}.$
\end{cor}
\begin{proof}
We know from 
   Theorem \ref{thm:IIM} that $A_{k+1}^{\mathrm{pol}}[Q_i^{-1}]$ is an integral domain, and hence 
   $A_{k+1}^{\mathrm{pol}}[Q_i^{-1}]$ is a $\C[Q]$-subalgebra  of 
   $(A_{k+1}^{\mathrm{pol}}[Q_i^{-1}])_S$, where $S=1+(Q).$
   Consider the composition of ring homomorphisms 
\begin{align*}
A_{k+1}^\pol[Q_i^{-1}] 
&\hookrightarrow (A_{k+1}^\pol[Q_i^{-1}])_S \\
&= ((A_{k+1}^\pol)_S)[Q_i^{-1}]\\
&\cong 
(QK^\pol(\Fl)_S)[Q_i^{-1}]\quad \text{(Proposition \ref{prop:KMmodified})}\\
&\hookrightarrow  
QK(\Fl)[Q_i^{-1}].
\end{align*}
The statement on $\G_w^Q$ follows from that of Proposition \ref{prop:KMmodified}.
\end{proof}

\end{rem}

\subsection{$K$-homology of the affine Grassmannian}\label{sec:Khom}
For $x\in \Wafo$, let $\xi_x^0\in K_*(\Gr)$ be the element
defined \cite[\S 6.3]{LSS}.
The non-equivariant $K$-theoretic $k$-Schur function $g_x^{(k)}$ in 
\cite[Theorem 7.17 (2)]{LSS} corresponds to $\xi_x^0$. 
Let us denote this isomorphism by $\alpha_{\circ}: 
K_*(\Gr)\rightarrow \Lambda_{(k)}.$
It holds that
$\alpha_{\circ}(\mathcal{O}_x^\Gr)=\gt{k}_x$ (\cite{Tak}, \cite[Lemma 2 (ii)]{LLMS}).
We define a twisted isomorphism $\alpha:
K_*(\Gr)\rightarrow \Lambda_{(k)}$ by
$\alpha:=\F^{-1}\circ \alpha_{\circ}.$
Thus we have
\begin{equation}
\alpha(\O_x^\Gr)=\F^{-1}(\gt{k}_x).\label{eq:alphaO}
\end{equation}

\begin{prop}\label{prop:di}
If $\lambda\subset R_i$ for some $1\leq i\leq k$, equivalently
$\lambda_1+\ell(\lambda)\leq k+1$, then
we have \begin{equation}
\alpha(\O_\lambda^\Gr)
=g_\lambda.
\end{equation}
In particular, we have 
\begin{equation}
\alpha(\O_{R_i}^\Gr)=\tau_i.
\end{equation}
\end{prop}
\begin{proof}
We know from Proposition 
\ref{prop:ksmallgtilde} that $\tilde{g}^{(k)}_\lambda=\tilde{g}_\lambda.
$
Therefore, using Proposition \ref{prop:Taki-g}, we see that
\begin{equation*}
\alpha(\O_\lambda^\Gr)
=\F^{-1}(\gt{k}_\lambda)
=\F^{-1}(\tilde{g}_{\lambda})
=g_\lambda.
\end{equation*}
\end{proof}

% \subsection{$K$-theoretic Peterson isomorphism}

% \begin{thm}
% For $\beta\in -Q^\vee$ such that $ wt_\beta\in \Wafo,$
% we have
% \begin{equation}
% \kappa\left(\mathcal{O}_{{wt_\beta}}^\Gr\cdot(\mathcal{O}_{t_\beta}^\Gr)^{-1}
% \right)= \mathcal{O}_{\Fl}^w\quad (w\in S_{k+1}).\label{eq:kappa}
% \end{equation}
% \end{thm}

\subsection{Correspondence of Schubert bases}

We consider the extended affine symmetric group $\Wafh$,
which is generated by $\{s_i\;|\;i\in I\}\cup\{\pi\}$ satisfying 
the same relations among $s_i$'s and
\begin{equation*}
\pi^{k+1}=id,\quad \pi s_i=s_{i+1}\pi.
\end{equation*}

We have $\hat{S}_{k+1}\cong S_{k+1}\ltimes P^\vee$, where  
$P^\vee$ is the coweight lattice of $\SL_{k+1}(\C)$.   
The translation element associated to $-\varpi_i^\vee\in P^\vee$ is explicitly given by  
\begin{equation*}
t_{-\varpi_i^\vee}=\pi^{-i}
x_{R_i}.
\end{equation*}
\begin{example}
For $k=3$,
\begin{eqnarray}
t_{-\varpi_1^\vee}=\pi^{-1}s_2s_3s_0,\quad
t_{-\varpi_2^\vee}=\pi^{-2}s_0s_3s_1s_0,\quad
t_{-\varpi_3^\vee}=\pi^{-3}s_2s_1s_0.
\end{eqnarray}
\end{example}

Let $w\in S_{k+1}$ be an $i$-{\it Grassmannian permutation\/}, i.e., $\Des(w)=\{i\}.$ The set of all $i$-Grassmannian permutations in $S_{k+1}$ is in bijection with $\Par_i^{k+1-i}.$
Explicitly, for an $i$-Grassmannian permutation $w$ in $S_{k+1}$,  
the corresponding partition $\lambda\in \Par_i^{k+1-i}$ is given by  
\begin{equation}
\lambda_{i+1-j}=w(j)-j\quad(1\leq j\leq i).\label{eq:shape}
\end{equation} 
For each partition $\lambda$ in $\Par_i^{k+1-i}$, we denote the corresponding
$i$-Grassmannian permutation by $w_{\lambda,i}. $ 
For $\lambda\in \Par^{k+1-i}_{i}$, 
the {\it dual\/} partition of $\lambda$ is the element
$\lambda^\vee$ 
in $\Par^{k+1-i}_{i}$ defined by $\lambda^\vee_{j}=k+1-i-\lambda_{i+1-j}\;(1\leq j\leq i).$
\begin{prop}[\cite{IIM}]\label{lem:Grass} 
Let $w_{\lambda,i}\in S_{k+1}$ be an $i$-Grassmannian permutation.
Then 
\begin{equation}
\Phi_{k+1}({\G_{\lambda,i}})
=\frac{g_{(\lambda^\vee)'}}{\tau_i},
 \label{eq:PhiGr}
\end{equation}
where $(\lambda^\vee)'$ is the conjugate of $\lambda^\vee$.
\end{prop}
\begin{proof} Let $\Phi_{k+1}'$ be the map in \cite{IIM}.
{Recall that for $w\in S_{k+1}$, $\O_{\Fl}^w$ is identified with $\mathfrak{G}_w^Q\;\mathrm{mod}\; I_{k+1}$. 
\cite[Theorem 7.1]{IIM} reads
$
\Phi_{k+1}'(\G_{w_{\lambda,i}}^Q \;\mathrm{mod}\; I_{k+1})={g_{\lambda^\vee}}/{g_{(k+1-i)^i}}.
$}
From Remark \ref{rem:omega}, we deduce \eqref{eq:PhiGr}. 
\end{proof}

In order to describe
the image $\Phi_{k+1}(\O^w_{\Fl})$,  
we need a map $S_{k+1}\rightarrow \Par^k,\;
w\mapsto \sh(w)$, due to Lam and Shimozono \cite[Lemma 11]{LS:MRL}.
\begin{prop}[\cite{LS:MRL}]\label{prop:LS11} Let $w\in S_{k+1}$. There is a $k$-bounded partition $\sh(w)$ such that 
\begin{equation}
wt_{-\sum_{i\in \Des(w)}\varpi_i^\vee}=
\pi^{-\sum_{i\in \Des(w)}i}
\cdot
x_{\sh(w)}.\label{eq:LS}
\end{equation}
\end{prop}
\begin{proof}
One can show that the left-hand side of \eqref{eq:LS} is an affine Grassmannian element in $\hat{S}_{k+1}=\langle \pi\rangle 
\ltimes \Waf$ (see the first part of the proof of Lemma 6.1 in \cite{LS:MRL}). Then 
such an element can be uniquely written in the form on the right-hand side of \eqref{eq:LS} for a $k$-bounded partition, which we denote by $\sh(w)$.  
\end{proof}

For an $i$-Grassmannian permutation $w=w_{\lambda,i}$, 
we have
$\sh(w)=(\lambda^\vee)'$ (\cite[Lemma 7.1]{IIM}).
See \S\ref{sec:Grass}, for another direct proof of this fact.

The next result is a refined version of \cite[Conjecture 1.8]{IIM}.

\begin{thm} For $w\in S_{k+1}$, we have
\begin{equation}
\Phi_{k+1}(\G_w)=\frac{\F^{-1}(\gt{k}_{\sh(w)})}{\prod_{i\in \Des(w)}\tau_i}.\label{eq:main}
\end{equation}
\end{thm}

\begin{proof}
Let $Q^\vee $ denote the coroot lattice of $\SL_{k+1}(\C).$ 
Then we have $\Waf\cong S_{k+1}\ltimes Q^\vee$. 
We denote by $t_\beta\in \Waf$ be translation element corresponding to $\beta\in Q^\vee.$ 
Let $K_*(\Gr)_\loc$ be the localization of $K_*(\Gr)$ by 
the multiplicative set generated by $\O_{t_\beta}^\Gr\;(\beta\in {Q}^\vee).$

{
Kato \cite[Corollary 4.15]{Kato} constructed an injective ring homomorphism
$$
\Kato:
K_*(\Gr)_\loc\hookrightarrow {QK}(\Fl)_\loc
$$
such that 
\begin{equation}
\kappa\left(\mathcal{O}_{{wt_\beta}}^\Gr\cdot(\mathcal{O}_{t_\gamma}^\Gr)^{-1}
\right)= Q^{\beta-\gamma}\mathcal{O}_{\Fl}^w\quad (w\in S_{k+1}),\label{eq:kappa}
\end{equation}
where $\beta,\gamma\in -Q^\vee$ such that $ wt_\beta\in \Wafo$ and $\gamma$ is strictly antidominant, and $Q^{\beta-\gamma}$ is defined by identifying 
$Q^{\alpha_i^\vee}$ with $Q_i.$}

Since $\O_{R_i}^\Gr$ corresponds to $\tau_i=g_{R_i}$ by $\alpha$ (Proposition \ref{prop:di}),
the isomorphism $\alpha$ yields
$
K_*(\Gr)_\loc
\cong
\Lambda_{(k)}[\tau_i^{-1}\;(1\leq i\leq k)].
$
Let $\Kato'$ be the composition
\begin{align*}
\Kato': 
K_*(\Gr)_{\loc}
&\cong
\Lambda_{(k)}[\tau_i^{-1}]\quad \text{(induced by $\alpha$)}\\
&\hookrightarrow 
\Lambda_{(k)}[\tau_i^{-1},(\tau_i^+)^{-1}]
\\
&\underset{\Phi_{k+1}^{-1}}{\cong} A_{k+1}^\pol[Q_i^{-1}] \quad\text{(Theorem \ref{thm:IIM})}\\
 &\underset{\phi}{\hookrightarrow }
QK(\Fl)[Q_i^{-1}]\quad\text{ (Corollary \ref{cor:phi})}.
\end{align*}

We claim that $\Kato'=\Kato.$ 
Note that $K_*(\Gr)$ is generated by $\O_{s_{i-1}\cdots s_1s_0}^\Gr=\O_{(i)}^\Gr\;(1\leq i\leq k),$
and that $\alpha(\O_{(i)}^\Gr)=g_{(i)}.$
For $1\leq i\leq k$, let $u_i\in S_{k+1}$ be the $k$-Grassmannian permutation of shape $(1^{k-i})$.
So by Proposition \ref{lem:Grass}, %\cite[Theorem 1.7]{IIM}, 
we have
$$
\Phi_{k+1}(\G_{u_i}^Q)
=\frac{g_{(i)}}{\tau_k}.
$$
Hence it follows that 
$$
\Kato'(\O_{(i)}^\Gr\cdot (\O_{R_k}^\Gr)^{-1})%=\Kato'(g_i)/\Kato'(\tau_j)
=\phi(\G^Q_{u_i})
=\O_{G/B}^{u_i}.
%\O^{u_i}_{\Fl}.
$$
By Lemma \ref{lem:lambdaGr}, we have 
$\sh(u_i)=(i)$.
Therefore, we have $\O_{u_i t_{-\varpi_k^\vee}}^\Gr=\O_{(i)}^\Gr$, and hence  by the definition of $\Kato$ 
$$
\Kato(\O_{(i)}^\Gr\cdot (\O_{R_k}^\Gr)^{-1})=\O^{u_i}_{\Fl}.
$$
Since $(\O_{R_k}^\Gr)^{-1}$ is an invertible element,
we have $\Kato'(\O_{(i)}^\Gr)=\Kato(\O_{(i)}^\Gr),$
and hence $\Kato'=\Kato.$  

For $w\in S_{k+1}$,
we deduce that $\O_{wt_{-\sum_{i\in \Des(w)}\varpi_i^\vee}}^\Gr=\O_{\sh(w)}^\Gr$ by Proposition \ref{prop:LS11},
and hence see that \begin{eqnarray*}
\phi(\G_w^Q)=
\O^w_{\Fl}%=
&=&\Kato\left(\O_{\sh(w)}^\Gr\cdot
\prod_{i\in \Des(w)}(\O_{R_i}^\Gr)^{-1}\right)\\
&=&\Kato'\left(\O_{\sh(w)}^\Gr\cdot
\prod_{i\in \Des(w)}(\O_{R_i}^\Gr)^{-1}\right)\\
&=&(\phi\circ\Phi_{k+1}^{-1}
)\left(\F^{-1}(\tilde{g}_{\sh(w)})\cdot \prod_{i\in \Des(w)}\tau_i^{-1}\right) \quad(\text{by \eqref{eq:alphaO}}).
\end{eqnarray*}
Since $\phi$ is injective we obtain \eqref{eq:main}.
\end{proof}

\subsection{Localizations of $K_*(\Gr)$}
The isomorphism conjectured by Lam, Li, Mihalcea, and Shimozono in \cite{LLMS} is 
different from $\Phi_{k+1}$ in the way of 
localization of $K_*(\Gr).$ The localization of 
$K_*(\Gr)$ in \cite{LLMS} can be identified with 
$\Lambda_{(k)}[\tau_i^{-1}\;(1\leq i\leq k)],$
while our version is 
$\Lambda_{(k)}[\tau_i^{-1},\;(\tau_i^+)^{-1}\;(1\leq i\leq k)].$ 
The aim of this section is to clarify the geometric meaning of $\tau_i^+=\tilde{g}_{R_i} \in \Lambda_{(k)}$.
\begin{lem}\label{lem:disToda}
$\tau_i^2-\tau_{i-1}\tau_{i+1}
=\tau_i^{+}\cdot\tau_i^{-}.$
\end{lem}
\begin{proof}
This is the discrete Toda equation given by Hirota~\cite{Hirota}.
We can show this by comparing the construction of $\Phi_{k+1}$  \cite{IIM} and the Lax formalism for the discrete Toda equation given in~\cite[Section 1]{IN}.
\end{proof}

\begin{lem} For $1\leq i\leq k$,
$
\Phi_{k+1}(1-Q_i)
=\dfrac{\tau_i^{+}\cdot\tau_i^{-}}{\tau_i^2}.
$
\end{lem}
\begin{proof}
This follows from Lemma \ref{lem:disToda}.
\end{proof}

\begin{lem}\label{eq:discToda} For $1\leq i\leq k+1$,
\begin{equation*}
\tau_i=g_{R_i^*}+\tau_i^{-}.
\end{equation*}
\end{lem}
\begin{proof} It suffices to show that $\F(\tau_i-g_{R_i^*})=\tau_i.$
This follows from \eqref{eq:Taksum}, because $R_i^*$ is the unique maximal proper 
element among the partitions 
$\mu\subset R_i.$ 
\end{proof}

\begin{prop}
For $1\leq i\leq k$, $$
({1-Q_i})z_1\cdots z_i=
1-\mathcal{O}_{\Fl}^{s_i}
$$
in $QK(G/B).$
\end{prop}
\begin{proof} By $\Phi_{k+1}(\mathcal{O}^{s_i})=g_{R_i^*}/\tau_i$ (\cite[Theorem 1.7]{IIM}), 
and Lemma \ref{eq:discToda},
the image of the right hand side is
$$
%\frac{\tau_i^2}{\tau_i^{+}\cdot\tau_i^{-}}
1-\frac{g_{R_i^*}}{\tau_i}
=\frac{\tau_i-g_{R_i^*}}{\tau_i}
=\frac{\tau_i^{-}}{\tau_i}.
$$
On the other hand,
$$
\Phi_{k+1}((1-Q_i)z_1\cdots z_i)=\frac{\tau_i^{-}\tau_i^{+}}{\tau_i^2}
\prod_{j=1}^i\frac{\tau_j\tau_{j-1}^{+}}{\tau_j^{+}\tau_{j-1}}
=\frac{\tau_i^{-}}{\tau_i}.
$$
\end{proof}
\begin{rem}
The result above corresponds to 
Corollary 3.33 in \cite{LM};
note also that
\begin{equation*}
1-\mathcal{O}_{\Fl}^{s_i}=\mathcal{O}_{\Fl}(-\varpi_i).
\end{equation*}
% See \cite{MNS} for more details.
\end{rem}
\begin{cor} The following element 
$$Q_{t_{\varpi_i^\vee}}^{-1}(1-Q_i)^{-1}(1-\O^{s_i}_{\Fl})$$
in $QK(G/B)_\loc$ is 
sent to $1/\tau^{+}_i$
by $\Phi_{k+1},$
where $Q_{t_{\varpi_i^\vee}}$ is the element such that $\Phi_{k+1}(Q_{t_{\varpi_i^\vee}})=\tau_i.$
\end{cor}
Note that the factor 
$Q_{t_{\varpi_i^\vee}}$ is invertible
in $QK(\Fl)_\loc$. Hence up to an invertible factor,
$1/\tau_i^+$ 
corresponds to the
element $(1-Q_i)^{-1}(1-\O^{s_i}_{\Fl})$
of $QK(\Fl)\subset QK(\Fl)_\loc.$

\appendix
\section{Parabolic quotient of Coxeter groups}\label{app:A}

We discuss some properties of a coset space of a Coxeter group. 
Our basic reference is Bjorner-Brenti \cite{BB}.
Let $(W, S)$ be a {\it Coxeter system} (see \cite{BB} for the definition), where $W$ is the group generated by $S=\{s_i\;|\;i\in I\}$
with index set $I$. 
The {\it Bruhat order} on $W$ (see  \cite[Chapter 2]{BB}) is 
denoted by $\leq.$
Let $J$ be any subset of $I$.
Let  $W_J$ be the subgroup of $W$ generated by $s_i\;(i\in J).$
The {\it minimal coset representatives\/} $W^J$ of the quotient $W/W_J$
is defined to be $W^J:=\{w\in W\;|\;ws_i>w\;\mbox{for all}\;i\in J\}$.
Any element $w$ of $W$ is expressed uniquely as $w=w^Jw_J$, with $w_J\in W_J$ and $w^J\in W^J$ (\cite[Proposition 2.4.4]{BB}). 

The following result
 is well-known and used throughout this section.
\begin{lem}\label{lem:WJ}
Let $x\in W^J$ and $i\in I$. Then, 
$s_ix <x\Longrightarrow s_i x\in W^J$.
\end{lem}

\subsection{Proof of Proposition \ref{prop:Tgcirc}}\label{sec:0Hmod}

The $0$-{\it Hecke algebra} $H_W$ is the associative $\C$-algebra generated by 
$\{T_i\;|\;i\in I\}$ subject to the same relations as those for $W$ except $T_i^2=-T_i$ in place of $s_i^2=id$. For $w\in W$, define $T_w=T_{i_1}\cdots T_{i_m}$ for any reduced expression $w=s_{i_1}\cdots s_{i_m}.$
The elements $T_w\;(w\in W)$ form a basis of $H_W.$

\begin{prop}\label{prop:VJ} Let $V^J$ be a left $H_W$ module given by 
\begin{equation*}
V^J=H_We_J,\quad
e_J:=\sum_{w\in W_J}T_w.
\end{equation*}
Then, $V^J=\bigoplus_{x\in W^J}\C a_x$
with  
$a_x:=T_xe_J$, and for $i\in I$,
\begin{equation}
T_i\cdot a_x=\begin{cases}
a_{s_i x}& (s_ix>x\;\mbox{and}\;
s_i x\in W^J)\\
-a_x& (s_ix<x)\\
0 & (s_ix>x\;\mbox{and}\;
s_i x\notin W^J)
\end{cases}.\label{eq:Ta}
\end{equation}
\end{prop}
\begin{proof}
For $i\in J$, we will prove $T_i e_J=0.$
Let 
$X^+=\{v\in W\;|\;s_iv>v\}$ and $X^-=\{v\in W\;|\;s_i v<v\}$.
We have $s_i(X^\pm)=X^{\mp}.$
If $v\in X^+$, then $T_iT_v=T_{s_iv}$.
If $v\in X^-$, then $T_{v}=T_i T_{s_iv}$, and hence
$T_i T_v=T_i^2 T_{s_iv}=-T_iT_{s_i v}=-T_{v}.$
Therefore, 
\begin{equation*}
T_ie_J
= \sum_{v\in X^+}
T_i T_v
+\sum_{v\in X^-}
T_i T_v
= \sum_{v\in X^+}
T_{s_iv}
-\sum_{v\in X^-}
T_{v}
=0.
\end{equation*}

If $v\notin W^J$, then it is easy to see that there is 
a reduced expression $v=s_{i_1}\cdots s_{i_m}$ with $i_m\in J$.
Hence $T_ve_J=T_{s_{i_1}\cdots s_{i_{m-1}}}T_{i_m}e_J=0$.
It follows that $\{a_x\}_{x\in W^J}$ 
spans $V^J=H_{W}e_J$. Note that for $x\in W^J$ and $v\in W_J$, we have
$T_xT_v=T_{xv}$ since $\ell(xv)=\ell(x)+\ell(v)$ (\cite[Proposition 2.4.4 (2)]{BB}); the linear independence of $\{a_x\}_{x\in W^J}$ follows from this fact.

Let $x\in W^J$.
If  $s_i x>x$, then 
\begin{equation*}
T_i\cdot a_x=T_i\cdot T_xe_J=T_{s_i x}e_J
=\begin{cases}
a_{s_ix}& (s_ix\in W^J)\\
0 & (s_i x \notin W^J)
\end{cases}.
\end{equation*}
If $s_i x<x$, then $T_i\cdot a_x=T_iT_xe_J
=-T_xe_J=-a_x.$
\end{proof}

\begin{prop}\label{prop:Db}
For $x\in W^J$, let $b_x:=\sum_{y\leq x}a_y.$
Set $D_i=T_i+1\;(i\in I)$.
Then\begin{equation}
D_i\cdot b_x=\begin{cases}
b_{s_i x}& (s_i x\in W^J,\;s_ix>x)\\
b_x&(\mbox{otherwise})
\end{cases}.\label{eq:Db}
\end{equation}
\end{prop}

\begin{lem}[$Z$-lemma]\label{lem:Z}
Let $w,v\in W$ and $i\in I$.
Suppse $s_iw>w$ and $s_iv>v$.
Then the following conditions are equivalent:
$$
(1)\;w\leq v,
\quad 
(2)\;s_i w\leq s_iv,
\quad
(3)\;w\leq s_iv.
$$
\end{lem}
\begin{proof}
\cite[Proposition 5.4.3]{MP}.
\end{proof}

\begin{lem}\label{lem:Z1}
Let $x\in W^J$ and $i\in I.$  We set
\begin{eqnarray*}
X_{\leq x}^{+}&:=&\{
y\in W^J\;|\;y\leq x,\;s_iy\in W^J,\;s_iy>y
\},\\
X_{\leq x}^{-}&:=&\{
y\in W^J\;|\;y\leq x,\;s_iy\in W^J,\;s_iy<y
\},\\
X^{0}_{\leq x}&:=&\{
y\in W^J\;|\;y\leq x,\;s_iy\notin W^J
\}.
\end{eqnarray*}

\begin{itemize}
\item[(1)] If $s_ix\in W^J,\;s_ix>x$, then
$s_i(X^+_{\leq x})=
X^{-}_{\leq s_i x},\;
X^+_{\leq x}=X_{\leq s_i x}^+,\;
X^0_{\leq x}=X_{\leq s_i x}^0.$
\item[(2)] If $s_ix\in W^J,\;s_i x<x$, then $s_i(X_{\leq x}^+)=X_{\leq x}^-.$
\item[(3)] If $s_ix\notin W^J$, then $s_i(X_{\leq x}^+)=X_{\leq x}^{-}.$
\end{itemize}
\end{lem}
\begin{proof} (1) and (2) follow immediately from Lemma \ref{lem:Z}. 
(3) The inclusion $s_i(X_{\leq x}^+)\supset X_{\leq x}^{-}$ follows from Lemma \ref{lem:Z}.
We will show that $s_i(X_{\leq x}^+)\subset X_{\leq x}^{-}.$
Take arbitrary $z\in s_i(X_{\leq x}^+)$. Write $z=s_iy$ with $y\in X_{\leq x}^+.$ Since $s_iz=s_i^2y=y<s_iy=z$,
it suffices to prove $s_iy \leq x.$ 
Since $x\in W^J$ and $s_i x\notin W^J$, we have $s_i x>x$
by Lemma \ref{lem:WJ}. Hence it follows from
Lemma \ref{lem:Z} that $s_i x\ge s_i y$. Also, it follows from
\cite[Corollary 2.5.2]{BB} that
% From Lemma \ref{lem:Z}, we have $s_ix\geq s_iy$. 
% Since $s_i x\notin W^J$, we have $s_ix>x$ by Lemma \ref{lem:WJ}. 
% Then 
there is $v\in W_J$ such that $s_ix=xv$ %(\cite[Corollary 2.5.2]).
; in fact, we can take $v=s_j $ for some $j\in J.$
Hence we have $xv\geq s_iy$.
Here note that $s_iy\in W^J$ as $y\in X_{\leq x}^+$.
It follows that 
$$
s_iy=(s_iy)^J
\leq (xv)^J
=x^J=
x;
$$
here we used the fact that $w\leq v$ for $w,v\in W$ implies
$w^J\leq v^J$ 
(\cite[Proposition 2.5.1]{BB}).
\end{proof}

\begin{proof}[Proof of Proposition \ref{prop:Db}]
We have $b_x=\sum_{y\in X_{\leq x}^+}a_y
+\sum_{y\in X_{\leq x}^-}a_y
+\sum_{y\in X_{\leq x}^0}a_y$.
Using \eqref{eq:Ta}, it is straightforward to verify 
\begin{eqnarray*}
D_i\cdot b_x%&=&T_i\b_z+\b_z\\
&=&
\sum_{y\in X_{\leq x}^+}a_{s_iy}
+\sum_{y\in X_{\leq x}^+}a_y
+\sum_{y\in X_{\leq x}^0}a_y.
\end{eqnarray*}
Consider first the case that $s_ix\in W^J,\;s_ix>x.$
From Lemma \ref{lem:Z1} (1), we see that
\begin{eqnarray*}
D_i \cdot b_x
&=&
\sum_{y\in X_{\leq s_ix}^-}a_{y}
+\sum_{y\in X_{\leq s_ix}^+}a_y
+\sum_{y\in X_{\leq s_ix}^0}a_y
=b_{s_i x}.
\end{eqnarray*}
Next consider the case that $s_i x\in W^J, s_ix<x$. We see that 
$s_i(X^+_{\leq x})=X_{\leq x}^-$ by Lemma \ref{lem:Z1} (2), and hence we see that  
\begin{equation*}
D_i\cdot b_x=\sum_{y\in X_{\leq x}^-}a_y
+\sum_{y\in X_{\leq x}^+}a_y
+\sum_{y\in X_{\leq x}^0}a_y=b_x.
\end{equation*}
Finally we consider the case that $s_i x\notin W^J$. We have 
$s_i(X^+_{\leq x})=X_{\leq x}^-$ by Lemma \ref{lem:Z1} (3), 
and hence $D_i\cdot b_x=b_x$ by exactly the same reasoning as in the previous case.
\end{proof}

\begin{proof}[Proof of Proposition \ref{prop:Tgcirc}]
We apply Proposition \ref{prop:VJ} to 
 $W=\Waf=\langle s_i\;|\;i=0,1,\ldots,k\rangle$, 
$W_J=S_{k+1},\; W^J=\Wafo$ with $J=\{1,2,\ldots,k\}$, 
and define an isomorphism 
$V^J\rightarrow \Lambda_{(k)}$ of vector spaces by 
$a_x\mapsto \gcirc{k}_\lambda, b_x\mapsto\tfg{k}_\lambda$
under the bijection $W^J=\Wafo \ni x\mapsto \lambda\in \Par^k$. 
Then we obtain Proposition \ref{prop:Tgcirc}.
\end{proof}

\subsection{Proof of Lemma \ref{lem:Naito2++}}\label{sec:Naito2}
For $w\in W$, and $i\in I$, we define $s_i*w$ 
in the same way as in \eqref{eq:*}.
Let $\A=(i_1,\ldots,i_r)$ be a 
 sequence 
 of elements of $I$, and $w\in S$. We define
\begin{equation}
\A*w=s_{i_1}*(s_{i_2}*\cdots *(s_{i_r}*w)\cdots).\label{eq:A*}
\end{equation}
\begin{prop}\label{prop:generalNaito2}
Let $\A=(i_1,\ldots,i_r)$ be a 
 sequence of elements of $I.$ 
Set $|\A|=r$, and $T_\A:=T_{i_1}\cdots T_{i_r}
\in H_W.$ Then for $x\in W^J$, we have
\begin{equation*}
T_{\A} \cdot a_x
=\begin{cases}
(-1)^{|\A|-\ell(\A*x)+\ell(x)}a_{\A*x} &
(\A*x\in W^J)
\\
0 & (\A*x\notin W^J)
\end{cases}.
\end{equation*}
\end{prop}
\begin{proof}
For $1\leq p\leq r$, set $\A_{p}=(i_p,\ldots,i_r)$, and $\A_{r+1}=\varnothing.$
We use decreasing induction on $p$. 
If $p=r+1$, then the assertion is obvious. Suppose $1\leq p<r+1$.
We first consider the case that $\A_p*x\notin W^J.$
Let $q\; (\geq p)$ be the maximal integer such that $\A_{q}*x\notin W^J.$
Then, we have $\A_{q+1}*x\in W^J$, $\A_{q}*x=s_{i_q}*(\A_{q+1}*x)=s_{i_q}(\A_{q+1}*x)\notin W^J$, and so 
$s_{i_q}(\A_{q+1}*x)>\A_{q+1}*x$ by Lemma \ref{lem:WJ}.  
Therefore, $T_{i_q}\cdot a_{\A_{q+1}*x}=0.$
By the inductive hypothesis,
we deduce that $T_{i_{q+1}}\cdots T_{i_r}\cdot a_{x}=\pm a_{\A_{q+1}*x}$, and hence
$$
T_{\A_{p}}\cdot a_x
=T_{i_p}\cdots T_{i_q}(T_{i_{q+1}}\cdots T_{i_r}\cdot a_{x})
=\pm T_{i_p}\cdots T_{i_q}\cdot a_{\A_{q+1}*x}=0.
$$ 

Next we consider the case when $\A_p*x\in W^J.$
Note that, in view of Lemma \ref{lem:WJ}, 
we have $\A_{q}*x\in W^J$ for $p\leq q\leq r+1$. 
By the inductive hypothesis, we have
$T_{i_{p+1}}\cdots T_{i_r}a_x=(-1)^{r-p-\ell(\A_{p+1}*x)+\ell(x)}\cdot a_{\A_{p+1}*x}$, so
\begin{equation}
T_{\A_p}\cdot a_x=T_{i_p}\cdot(T_{i_{p+1}}\cdots T_{i_r} a_x)
=(-1)^{r-p-\ell(\A_{p+1}*x)+\ell(x)}
T_{i_p}\cdot a_{\A_{p+1}*x}.\label{eq:hence}
\end{equation}
Now we consider two cases : (a) $s_{i_p}(\A_{p+1}*x)>\A_{p+1}*x$, (b) $s_{i_p}(\A_{p+1}*x)<\A_{p+1}*x.$
If (a) holds, then $s_{i_p}(\A_{p+1}*x)=s_{i_p}*(\A_{p+1}*x)=\A_{p}*x\in W^J.$ Therefore,
$$
T_{i_p}\cdot a_{\A_{p+1}*x}=a_{s_{i_p}(\A_{p+1}*x)}=a_{\A_{p}*x},
$$
and hence by \eqref{eq:hence}, 
$$
T_{\A_p}\cdot a_x
=(-1)^{r-p-\ell(\A_{p+1}*x)+\ell(x)}
T_{i_p}a_{\A_{p+1}*x}
=(-1)^{r-p-\ell(\A_{p+1}*x)+\ell(x)}a_{\A_{p}*x};
$$
the sign is correct 
since we have $\ell(\A_{p+1}*x)=\ell(\A_{p}*x)-1.$ 
Suppose (b) holds. 
Note that $\A_{p}*x=\A_{p+1}*x\in W^J$ by \eqref{eq:A*}. By the definition of $T_{i_p}$,
$$T_{i_p}\cdot a_{\A_{p+1}*x}=-a_{\A_{p+1}*x}=-a_{\A_{p}*x}.
$$ 
Hence we deduce that 
$$
T_{\A_p}\cdot a_x
=(-1)^{r-p-\ell(\A_{p+1}*x)+\ell(x)}
(-a_{\A_{p}*x})
=(-1)^{r-p+1-\ell(\A_{p}*x)+\ell(x)}a_{\A_p*x},
$$
where we again used $\A_{p}*x=\A_{p+1}*x.$
This completes the proof.
\end{proof}

\begin{proof}[Proof of Lemma \ref{lem:Naito2++}]
Let $A\subsetneq I$, with $|A|=r$, and take a reduced expression  
$s_{i_1}\cdots s_{i_r}$ for $u_A$.
We apply Proposition \ref{prop:generalNaito2} to the sequence $\A=(i_1,\ldots,i_r)$.
Then $T_\A=T_{u_A}$ and we obtain Lemma \ref{lem:Naito2++}.
\end{proof}

\section{Grassmannian permutations}\label{sec:Grass}

An explicit description of $\sh(w)$ is available (see \cite[\S 6]{LS:MRL}). 
For the reader's convenience, we include a simple direct proof 
when $w$ is a Grassmannian element $w_{\lambda,i}$ (Lemma \ref{lem:lambdaGr}).

\begin{prop}\label{prop:afGr}
Let $\lambda\in \Par^k$ be a $k$-bounded partition of size $r$ such that $\lambda_1+\ell(\lambda)\leq k+1$.
Take any standard tableau $T$ of shape $\lambda$. 
We denote the box of $\lambda$ with entry $i$ in $T$ by $b_T(i).$
Then $s_{\res(b_T(r))}\cdots s_{\res(b_T(2))}s_{\res(b_T(1))}$ is a reduced 
expression for $x_{\lambda}$. 
\end{prop}
\begin{proof}
We note first that the corresponding fact
is well-known for the $i$-Grassmannian permutation $w_{\lambda,i}$.
The reader can consult \cite[\S3.1]{BCMP(Chev)}
for an exposition of this fact in a more general setting. 

Since $\lambda_1+\ell(\lambda)\leq k+1$, there is 
$1\leq i\leq k$ such that 
$\lambda\subset R_{k+1-i}$. Then in the reduced 
expression of $x_\lambda$ given by \eqref{eq:redexp}, $s_{k+1-i}$ does not appear.  
Hence $x_\lambda\in \langle s_{-i+1},\ldots,s_{-i+k}\rangle
\cong S_{k+1};$
the isomorphism of groups $\phi:S_{k+1}\rightarrow  \langle s_{-i+1},\ldots,s_{-i+k}\rangle$ is given by $\phi(s_j)=s_{-i+j}$.
 Then, the result follows from the case for the $i$-Grassmannian permutation.
\end{proof}
\begin{rem}\label{rem:fulcom}
An element in a Coxeter group $W$ is {\it fully commutative\/} if 
 any two of its reduced expressions are related by a series of transpositions of adjacent commuting generators.
% In particular, all the reduced expression $s_{i_1}\cdots s_{i_r}$ of a fully commutative element
% have the same indices $i_1,\ldots,i_s.$
It is well-known that any $i$-Grassmannian element in $S_{k+1}$ is 
fully commutative. 
From the above proof, we see that if $\lambda_1+\ell(\lambda)\leq k+1$, then 
$x_\lambda$ is fully commutative.
\end{rem}

\begin{lem}\label{lem:lambdaGr}
Let $\lambda$ be a partition contained in $R_{k+1-i}=(k+1-i)^i$, and $w_{\lambda,i}\in S_{k+1}$ the corresponding 
$i$-Grassmannian permutation. Then $\sh(w_{\lambda,i})=(\lambda^\vee)'.$ \end{lem}
\begin{proof}In the extended affine symmetric group
$\hat{S}_{k+1}$, we compute
$w_{\lambda,i}t_{-\varpi_i^\vee}=w_{\lambda,i}\pi^{-i}x_{R_i}=\pi^{-i}(\pi^i w_{\lambda,i}\pi^{-i})x_{R_i}.$
A reduced expression for $y_\lambda:=\pi^i w_{\lambda,i}\pi^{-i}$ is 
obtained by replacing $s_j$ with $s_{j+i}$ in $w_{\lambda,i}$. %(Note $\phi^{-1}(s_j)=s_{j+i}$.) 
%In particular, note that 
%the residue at the box $(1,1)$ of $\pi^i w_{\lambda,i}\pi^{-i}$ is $2i$.
%Note also that the residue of the unique corner $(k+1-i,i)$ of $R_i$ is $2i.$ 
It is straightforward to see
when we reflect the tableau of $y_\lambda$ along the line with a slope of $1$, 
it fits inside the tableau of shape $R_i$ filled with $(k+1)$-residues adjusted to the south-east corner (see Example \ref{ex:Grass} below).  
{Let $T$ be the standard tableau with shape $R_{i}$ 
that is obtained by filling positive integers into each row of $(\lambda^\vee)'$ from 
left to right, with rows taken from top to bottom, and then into each column
of the remaining boxes from top to bottom, with columns taken from left to right.
Then apply Proposition \ref{prop:afGr} 
to $T$.}
The obtained reduced expression for  
$x_{R_i}$ shows $x_{R_i}=y_\lambda^{-1}\cdot x_{(\lambda^\vee)'},$
and hence $ w_{\lambda,i}t_{-\varpi_i^\vee}=\pi^{-i}x_{(\lambda^\vee)'}$, showing
$\sh(w_{\lambda,i})=(\lambda^\vee)'.$
\end{proof}

\begin{example}\label{ex:Grass}
For $k=6, i=3$, and $\lambda=(3,2)$, the corresponding $3$-Grassmannian
element
is $w_{\lambda,3}=s_3s_2\cdot s_5s_4s_3.$
%$: \ytableausetup{smalltableaux}
%\begin{ytableau}
%3 & 4 & 5\\
%2 &3
%\end{ytableau} .$ 
We have $$
y_{\lambda}:=\pi^3(w_{\lambda,3})\pi^{-3}=s_6s_5\cdot s_1s_0s_6
:\quad
\begin{ytableau}6 & 0 & 1\\ 5 &6
\end{ytableau}.
$$
We read the entries of 
the tableau of shape $R_3$ filled with $7$-residues
\begin{ytableau}
0 & 1 & 2\\ 6 & 0 & *(lightgray)1  \\5 & *(lightgray)6 &*(lightgray) 0 \\4& *(lightgray)5 &*(lightgray)6
\end{ytableau}
according to the order given by  $T=\begin{ytableau}
1 & 2 & 3\\ 4 & 5 &  *(lightgray)10  \\6 & *(lightgray)8  & *(lightgray)11 \\ 7&*(lightgray) 9&*(lightgray)12
\end{ytableau}$
to obtain
$$
x_{R_3}%=s_6s_5s_4 s_0s_6s_5 s_1s_0s_6s_2s_1s_0
=(s_6s_0s_1\cdot s_5s_6 )\cdot s_4s_5s_0s_6 s_2s_1s_0.
%\begin{ytableau}
%0 & 1 & 2\\ 6 & 0 &  *(gray)1  \\5 & *(gray)6  & *(gray)0 \\ 4& *(gray)5 &*(gray)6
%\end{ytableau}.
$$
This is $y_\lambda^{-1}x_{(\lambda^\vee)'}$.
%
%\ytableausetup{smalltableaux}
%$$w_{(3,1),3}\mapsto \begin{ytableau}
%3 & 4 & 5\\ 2\end{ytableau},\quad
%x_{R_3}\mapsto 
%\Tableau{0 & 1 & 2\\ 6 & 0 &  1  \\5 & 6 & 0 \\ 4& 5 &6}$$
%and
%$$
%$$
The shaded boxes correspond to $y_\lambda$.
\end{example}

\section{Vertical Pieri rule for the closed $K$-$k$-Schur functions}\label{app:C}
A Pieri rule for $\gt{k}_\lambda$ was proved by Takigiku \cite{Tak1}.
We record here the vertical version of Takigiku's formula, 
which should be known to experts but is missing in the literature. 
We do not use this result in the main part of this paper. 

For $\lambda\in \Par^k$, 
$1\leq r\leq r,$ define
\begin{eqnarray*}
\mathcal{H}_{\lambda,r}^{(k)}
&:=&\{A\subsetneq I\;\left|\; |A|=r,\;d_Ax_\lambda\in \Wafo,\;d_Ax_\lambda\geq_L x_\lambda\}\right.,\\
\mathcal{V}_{\lambda,r}^{(k)}
&:=&\{A\subsetneq I\;\left|\; |A|=r,\;u_Ax_\lambda\in \Wafo,\;u_Ax_\lambda \geq_L x_\lambda\}\right..
\end{eqnarray*}
For an element $A$ in 
$\mathcal{H}_{\lambda,r}^{(k)}$, we write $d_Ax_\lambda=x_{\kappa}$ for 
an element $\kappa\in \Par^k,$ and denote this $\kappa $ by $d_A\lambda$.
Similarly, for an element $A$ in 
$\mathcal{V}_{\lambda,r}^{(k)}$, we write $u_Ax_\lambda=x_{\kappa}$ for 
an element $\kappa\in \Par^k,$ and denote this $\kappa $ by $u_A\lambda$.

\begin{example}\label{ex:C1} For $k=3$, $\lambda=(2,1)\in \Par^3,$
and $r=2$, we have (see Figure \ref{fig:21})
$$\mathcal{H}_{\lambda,r}^{(k)}=\{\{2,3\},\{0,2\}\},\quad
d_{\{2,3\}}=s_3s_2,\quad
d_{\{0,2\}}=s_0s_2=s_2s_0,
$$
$$
\mathcal{V}_{\lambda,r}^{(k)}=\{\{1,2\},\{0,2\}\},\quad
u_{\{1,2\}}=s_1s_2,\quad
u_{\{0,2\}}=s_0s_2=s_2s_0.
$$
The corresponding weak (horizontal and vertical) strips are given by 
$$
d_{\{2,3\}}\lambda%=s_3s_2\lambda
=(3,1,1),\quad
d_{\{0,2\}}\lambda=u_{\{0,2\}}
\lambda
=(2,1,1),\quad
u_{\{1,2\}}\lambda%=s_1s_2\lambda
=(2,1,1,1).
$$
\end{example}

For $A_1,\ldots,A_m\in \mathcal{H}_{\lambda,r}^{(k)}$, it is known by Takigiku \cite[Corollary 4.8]{Tak1} that 
$A_1\cap \cdots\cap A_m\in \mathcal{H}_{\lambda,r'}^{(k)}$, with 
$r'=|A_1\cap\cdots \cap A_m|$ and hence
$d_{A_1\cap\cdots\cap A_m}\lambda\in \Par^k$
is defined.

\begin{prop}\label{prop:barH}
Let $A\mapsto \overline{A}$ be 
the map given by sending $i\in A$ to $-i\in \overline{A}.$
Then for $\lambda\in \Par^k$, and $1\leq r\leq k,$
\begin{equation*}
\overline{\mathcal{H}_{\lambda,r}^{(k)}}=
\mathcal{V}_{\lambda^{\omega_k},r}^{(k)}.
\end{equation*}
\end{prop}
\begin{proof}
Let $A\in \mathcal{H}_{\lambda,r}^{(k)}$.
Recall that $\omega_k$ 
is an automorphism of the group $\Waf$ preserving the left weak order. Therefore, it follows that
$$
u_{\overline{A}}x_{\lambda^{\omega_k}}=
\omega_k(d_A)\omega_k(x_{\lambda})=
\omega_k({d_Ax_{\lambda}})
\geq_L
\omega_k(x_\lambda)
=x_{\lambda^{\omega_k}},
$$
and hence $\overline{A}\in \mathcal{V}_{\lambda,r}^{(k)}.$
Thus $\overline{\mathcal{H}_{\lambda,r}^{(k)}}\subset
\mathcal{V}_{\lambda^{\omega_k},r}^{(k)}$.
Similarly we have $\overline{\mathcal{H}_{\lambda,r}^{(k)}}\supset
\mathcal{V}_{\lambda^{\omega_k},r}^{(k)}$.
\end{proof}

\begin{thm}[\cite{Tak}]
Let $\lambda\in \Par^k$ and $1\leq r\leq k$.
Let 
\begin{equation*}
\mathcal{H}_{\lambda,r}^{(k)}=\{A_1,\ldots,A_m\},\quad
\mathcal{V}_{\lambda,r}^{(k)}=\{B_1,\ldots,B_n\}.
\end{equation*}
 Then
\begin{eqnarray*}
\tilde{g}_{(r)}\cdot\gt{k}_\lambda
&=&\sum_{i= 1}^m
(-1)^{i-1}
\sum_{1\leq a_1<\cdots<a_i\leq m}
\gt{k}_{d_{A_{a_1}\cap\cdots\cap A_{a_i}}\lambda},\label{eq:gtildePieri}\\
\tilde{g}_{(1^r)}\cdot\gt{k}_\lambda
&=&\sum_{i= 1}^n
(-1)^{i-1}
\sum_{1\leq a_1<\cdots<a_i\leq n}
\gt{k}_{u_{B_{a_1}\cap\cdots\cap B_{a_i}}\lambda}.\label{eq:gtildePieri2}
\end{eqnarray*}
\end{thm}
\begin{proof}
\eqref{eq:gtildePieri} is due to Takigiku \cite{Tak}.
From Proposition \ref{prop:barH}, it follows that 
$\{\overline{A_1},\ldots,\overline{A_m}\}
=\mathcal{V}_{\lambda^{\omega_k},r}^{(k)}.$
By applying $\Omega$ to both sides of the equation, we obtain  
$$
\tilde{g}_{(1^r)}\cdot\gt{k}_{\lambda^{\omega_k}}
=\sum_{i= 1}^m
(-1)^{i-1}
\sum_{1\leq a_1<\cdots<a_i\leq m}
\gt{k}_{u_{\overline{A_{a_1}}\cap\cdots\cap\overline{A_{a_i}}}\lambda^{\omega_k}},
$$
where we used 
$\Omega(\tilde{g}_{(r)})=\tilde{g}_{(1^r)}$ and 
 $$\Omega(\gt{k}_{d_{A_{a_1}\cap\cdots\cap A_{a_i}}\lambda})
=\gt{k}_{u_{\overline{A_{a_1}\cap\cdots\cap A_{a_i}}}\lambda^{\omega_k}}
=\gt{k}_{u_{\overline{A_{a_1}}\cap\cdots\cap \overline{A_{a_i}}}\lambda^{\omega_k}}.$$
Thus we  have \eqref{eq:gtildePieri2}.
\end{proof}

\begin{example}
According to Example \ref{ex:C1}, we have
\begin{eqnarray*}
\gt{3}_{(2)}
\gt{3}_{(2,1)}
&=&\gt{3}_{(3,1,1)}
+\gt{3}_{(2,2,1)}
-\gt{3}_{(2,1,1)},
\\
\gt{3}_{(1,1)}
\gt{3}_{(2,1)}
&=&\gt{3}_{(2,1,1,1)}
+\gt{3}_{(2,2,1)}
-\gt{3}_{(2,1,1)}.
\end{eqnarray*}

\ytableausetup{smalltableaux}

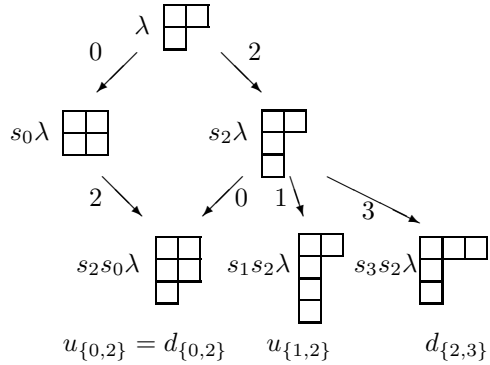
\begin{figure}[htbp]
\centering
\begin{picture}(150,120)
\put(58,120){\begin{ytableau}
{}&\\
{}\end{ytableau}} %%% (2,1)
\put(90,110){$2$}
\put(30,110){$0$}
\put(30,55){$2$}
\put(85,55){$0$}
\put(100,55){$1$}
\put(133,50){$3$}
\put(80,113){\vector(1,-1){15}}
\put(48,113){\vector(-1,-1){15}}
\put(20,80){\begin{ytableau}
{}&\\
{}&{}\end{ytableau}}
\put(35,66){\vector(1,-1){15}}
\put(95,76){\begin{ytableau}{}&\\
{}\\
{}\end{ytableau}}
\put(88,66){\vector(-1,-1){15}}
\put(106,66){\vector(1,-3){5}}
\put(120,66){\vector(2,-1){35}}
\put(55,28){\begin{ytableau}{}&\\
{}&{}\\
{}\end{ytableau}}
\put(109,25){\begin{ytableau}{}&\\
{}\\
{}\\
{}\end{ytableau}}
\put(155,28){\begin{ytableau}{}&&\\
{}\\
{}\end{ytableau}}
\put(20,0){$u_{\{0,2\}}=d_{\{0,2\}}$}
\put(97,0){$u_{\{1,2\}}$}
\put(157,0){$d_{\{2,3\}}$}
\put(47,120){$\lambda$}
\put(0,80){$s_0\lambda$}
\put(75,80){$s_2\lambda$}
\put(25,30){$s_2s_0\lambda$}
\put(82,30){$s_1s_2\lambda$}
\put(130,30){$s_3s_2\lambda$}
\end{picture}
\caption{$k=3$, $\lambda=(2,1)$, $r=2$.}
\label{fig:21}
\end{figure}
\end{example}
\bigskip

\bibliographystyle{amsplain}

\end{document}